\documentclass[A4]{amsart}
\usepackage{amsthm,amsmath,amssymb,eucal}
\usepackage{geometry}
\usepackage{resizegather}
\usepackage{graphicx}
\usepackage{color}
\usepackage[normalem]{ulem}
\usepackage[latin1]{inputenc}

\newtheorem{claim}{Claim}[section]
\newtheorem{theorem}[claim]{Theorem}

\newtheorem{remark}[claim]{Remark}

\newcommand{\soutg}{\bgroup\markoverwith{\textcolor{green}{\rule[.5ex]{2pt}{1pt}}}\ULon}
\newcommand{\soutb}{\bgroup\markoverwith{\textcolor{blue}{\rule[.5ex]{2pt}{1pt}}}\ULon}
\newcommand{\soutr}{\bgroup\markoverwith{\textcolor{red}{\rule[.5ex]{2pt}{1pt}}}\ULon}

\title[Platonic solid graphs]{Spectral asymptotics of the Laplacian \\[.1em] on Platonic solids graphs}

\author{Pavel Exner}

\address{{Doppler Institute for Mathematical Physics and Applied Mathematics, Czech Technical University,
B\v rehov{\'a} 7, 11519 Prague, Czechia} {\rm and} {Department of Theoretical Physics, Nuclear Physics Institute, Czech Academy of Sciences, 25068 \v{R}e\v{z} near Prague, Czechia}
}

\email{exner@ujf.cas.cz}

\author{Ji\v{r}\'{\i} Lipovsk\'{y}}

\address{{Department of Physics, Faculty of Science, University of Hradec Kr\'alov\'e, Rokitansk\'eho 62,
500\,03 Hradec Kr\'alov\'e, Czechia} {\rm and} {Department of Theoretical Physics, Nuclear Physics Institute, Czech Academy of Sciences, 25068 \v{R}e\v{z} near Prague, Czechia}}
\email{jiri.lipovsky@uhk.cz}

 \date{\today}

\begin{document}

\begin{abstract}
We investigate the high-energy eigenvalue asymptotics quantum graphs consisting of the vertices and edges of the five Platonic solids considering two different types of the vertex coupling. One is the standard $\delta$-condition, the other is the preferred-orientation one introduced in \cite{ETa}. The aim is to provide another illustration of the fact that the asymptotic properties of the latter coupling are determined by the vertex parity by showing that the octahedron graph differs in this respect from the other four for which the edges at high energies effectively disconnect and the spectrum approaches the one of the Dirichlet Laplacian on an interval.
\end{abstract}

\maketitle

%%%%%%%%%%%%%%%
\section{Introduction}
%%%%%%%%%%%%%%%

The quantum graph concept was proposed in the early days of quantum mechanics by Linus Pauling \cite{Pa} as a model of aromatic hydrocarbon molecules, and soon it was forgotten and rediscovered again only in the 1980s. Since then it is an object of intense interest revealing a number of interesting effects; we refer to the monograph \cite{BK} and the extensive bibliography it contains. One of the characteristic features of quantum graphs is possibility to choose different \emph{vertex couplings}, i.e. the conditions matching the wavefunctions at the graph nodes due to the fact that the self-adjointness requirement leaves a considerable freedom in the choice of the coupling.

Not surprisingly, most often one encounters simple matching conditions such as the $\delta$-coupling that requires the continuity of the function values at the vertex and the sum of the outgoing derivatives there to be equal to an $\alpha$-multiple of the common function value for a fixed $\alpha\in\mathbb{R}$. For its particular case with $\alpha=0$ the -- not quite fortunate\footnote{Recall that Kirchhoff law in electricity means the current conservation and that \emph{any} self-adjoint quantum graph dynamics preserves the probability current at the vertices. Various alternative names such as free, standard, or natural were proposed but none received a common recognition.} -- name \emph{Kirchhoff} is used; recall that these conditions arise in the limit of the Neumann Laplacian supported by a network which collapses to a graph \cite{RS, EPo}. However, in some situations other boundary condition may be useful.

Recently quantum graphs (with the $\delta$-coupling) were used to model the anomalous Hall effect \cite{SK}. To get the result, the authors had to suppose that the electron motion on the graph loops has a prescribed orientation which is an assumption hard to justify from the first principles. On the other hand, the vertex coupling \emph{can} be non-invariant with respect to time reversal. A simple example was proposed in \cite{ETa} and its investigation revealed an interesting \emph{topological property} of such a vertex coupling, namely that the transport properties of the vertex at high energies depend substantially on the \emph{vertex parity}; this effect was illustrated in \cite{ETa} through comparison of the band spectra of square and hexagonal lattices.

The main aim of the present paper is to demonstrate the same effect in the discrete spectrum of finite quantum graphs. We investigate five simple graphs associated with the five Platonic solids assuming the full symmetry, i.e. the same coupling at every graph vertex. We compare the high-energy behavior of the eigenvalues in two cases, the $\delta$-coupling and the preferred-orientation one from \cite{ETa}. The symmetry allows us to simplify considerably the spectral analysis by splitting the operator into components in eigenspaces of the corresponding symmetry group representation. This method was employed before -- recall Naimark and Solomyak reduction of homogeneous trees \cite{NS} or the proof of Thm. 7.3 in \cite{DEL} -- and the present discussion provides another illustration of its usefulness. We also note that examination of Platonic solid graphs resonates with the original Pauling's motivation because such geometric structures can be found in nature, in the shape of molecules -- methane for the tetrahedron, cubane for the cube, $\mathrm{SF}_6$ for the octahedron, dodecahedrane for the dodecahedron, or boron for the icosahedron -- or viruses, etc.

The results confirm the expectation. The graphs with the $\delta$-coupling have several series of eigenvalue, the number of which increases with the graph complexity. Some eigenvalues are of the `single-edge' type and do not depend on the interaction strength $\alpha$, the others do but in the high-energy region they are close to the eigenvalues referring to the Kirchhoff coupling; to be more specific, their square roots approach those of the Kirchhoff eigenvalues with an $\mathcal{O}(n^{-1})$ error which mean that the distance between the eigenvalues and their Kirchhoff counterpart has an $n$-independent bound which is in good correspondence with the results of \cite{KS}.

The spectral picture of the graphs with the preferred-orientation coupling is different. The high-energy behavior singles out the octahedron which has four series of eigenvalues, two of which determined exactly, while the elements of the others have asymptotically periodic square roots. The spectra of the other four Platonic solid graphs are simpler and asymptotically coincide with the spectrum of the Dirichlet Laplacian on an interval representing a single edge. The reason of the difference comes from the above mentioned topological property. The octahedron graph has vertices of degree four and at high energies the particle has roughly the same probability of being reflected at the vertex or transmitted to an adjacent edge. Vertices of the other four graphs, on the other hand, have an odd degree, three or five, and as a consequence, the on-shell scattering matrix at the vertex approaches the identity matrix as $k\to\infty$. In other words, at high energies the particle is almost surely reflected and remains thus confined to a single edge.

%%%%%%%%%%%%%%%
\section{Description of the model}
%%%%%%%%%%%%%%%
Let us briefly recall a few basic notions about equilateral quantum graphs referring to the monograph \cite{BK} for more details. A finite equilateral metric graph consists of a family of vertices and $N$ finite edges connecting them according to a prescribed adjacency matrix; to parametrize the edges we suppose that they are homothetic to the interval $(0,1)$. The operators $H$ we are interested in act as the negative second derivative on the graph edges, with the domain consisting of functions in the Sobolev space $\bigoplus_{j=1}^N W^{2,2}((0,1))$ satisfying at the $j$-th vertex the coupling conditions $(U_j-I)\Psi_j+i(U_j+I)\Psi_j'=0$, where $U_j$ is a unitary coupling matrix, $I$ is an identity matrix, $\Psi_j$ and $\Psi_j'$ are the vectors of limits of function values and outgoing derivatives at the given vertex, respectively. Note that the graphs we consider are not oriented; if a vertex corresponds to $x=0$, we take $f'(0)$ as the entry of $\Psi_j'$, if it refers to $x=1$ we take $-f'(1)$ instead.

The graphs we discuss in the present paper have edges and vertices of the Platonic solids. As indicated in the introduction we are going to compare such graph Hamiltonians with two particular couplings types. The first one is the $\delta$-coupling, which requires the continuity of the function values at the vertex, and the sum of the outgoing derivatives to be equal equals to the $\alpha$-multiple of this common function value. We consider the symmetric situation when the coupling constant $\alpha$ is the same at all the vertices. The matrix $U_j$ is then independent of $j$ and has the form $U_j = \frac{2}{d_j+i\alpha}J-I$, where $d_j$ is the vertex degree and the matrix $J$ has all the entries equal to one. The second type, dubbed \emph{preferred-orientation}, was proposed with the participation of one of the authors in \cite{ETa} as an example of a coupling not invariant with respect to the time reversal. It exhibits a rotational behavior which is `maximal' at a particular energy, here assumed conventionally to refer to $k=1$, when the incoming wave at an edge passes to just one neighboring edge and this happens in cyclical manner. The corresponding matrix $U_j$, identified with the on-shell S-matrix $S(1)$, has the form
 % -------------- %
$$
  U_j = \begin{pmatrix}0 & 1 & 0 & 0 & \cdots & 0 & 0\\ 0 & 0 & 1 & 0 & \cdots & 0 & 0\\ 0 & 0 & 0 & 1 & \cdots & 0 & 0\\ \vdots & \vdots & \vdots & \vdots & \ddots & \vdots & \vdots\\ 0 & 0 & 0 & 0 & \cdots & 0 & 1\\ 1 & 0 & 0 & 0 & \cdots & 0 & 0\\ \end{pmatrix}\,.
$$
 % -------------- %
For both types of the coupling we focus our attention at the high-energy behavior, that is, the asymptotical positions of the eigenvalues of $H$; as usual, we shall work with their square roots $k$ and examine the limit for $k\to \infty$.

%%%%%%%%%%%%%%%
\section{The tetrahedron}
%%%%%%%%%%%%%%%

 % -------------- %
\begin{figure}
\centering
\includegraphics[height=5cm]{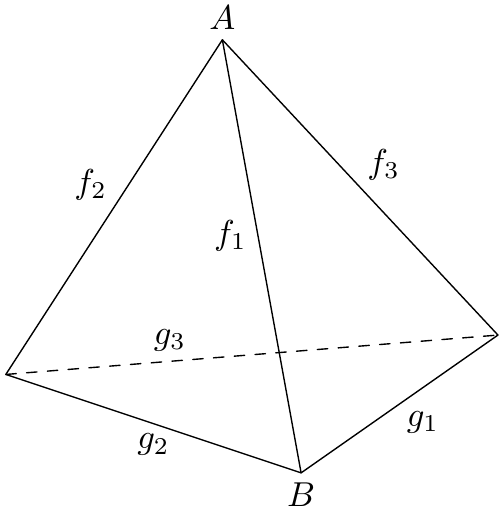}
\caption{A tetrahedron}
\label{fig1}
\end{figure}
 % -------------- %
Assume first the tetrahedron-shaped graph $\Gamma$ with all the edge lengths equal to one. Let us denote the wavefunction components by $f_1$, $f_2$, $f_3$, $g_1$, $g_2$, $g_3$ as indicated in~Figure~\ref{fig1}. We parametrize the edges by intervals $(0,1)$ in the following way. For $f_j$, $j=1,2,3$, the point $x=0$ is denoted by $A$. The point $B$ corresponds to $x=1$ for $f_1$, $x=0$ for $g_1$ and $x=1$ for $g_2$. We assume that the point $x=0$ for $g_3$ corresponds to $x=1$ for $g_1$. We consider an operator $T_3$ which acts on $\Gamma$ as a (clockwise, if viewed from the outside of the tetrahedron) rotation by $\frac{2\pi}{3}$ around the axis defined by the vertex $A$ and the center of the side opposite to it. Since $T_3^3 = \mathrm{Id}$, it has eigenvalues $\omega_j = \mathrm{e}^{\frac{2\pi i}{3}j}$, $j = 0,1,2$. We consider subspaces of eigenfunctions of the operator $T_3$, in which the following relations hold.
 % -------------- %
\begin{eqnarray}
  f_2(x) = \omega_j f_1(x)\,,\quad f_3(x) = \omega_j^2 f_1(x)\,,\label{eq:4sten01}\\
  g_2(x) = \omega_j g_1(x)\,,\quad g_3(x) = \omega_j^2 g_1(x)\,.\label{eq:4sten02}
\end{eqnarray}
 % -------------- %
We investigate separately the two types of the vertex coupling described above.

%%%%%%%%%%%%%%%
\subsection{$\delta$-condition}
%%%%%%%%%%%%%%%
We suppose that $\delta$-conditions of strength $\alpha$ are imposed at all the vertices. At the vertex $A$ we then have
 % -------------- %
\begin{equation}
  f_1(0) = f_2(0) = f_3(0)\,,\quad f_1'(0)+f_2'(0)+f_3'(0) = \alpha f_1(0)\,.\label{eq:4sten:delta01}
\end{equation}
 % -------------- %
For $j = 0$ we have $f_1(x) =f_2(x) = f_3(x)$, and thus
 % -------------- %
\begin{equation}
  f_1'(0) = \frac{\alpha}{3} f_1(0)\,.\label{eq:4sten:bc:1a}
\end{equation}
 % -------------- %
For $j = 1,2$, on the other hand, we have $f_2(0) = \omega_j f_1(0)$ and using the first equation in \eqref{eq:4sten:delta01} we obtain
 % -------------- %
\begin{equation}
  f_1(0) = 0\,.\label{eq:4sten:bc23a}
\end{equation}
 % -------------- %

At the vertex $B$ the coupling condition is
 % -------------- %
$$
  f_1(1) = g_1(0) = g_2(1)\,,\quad -f_1'(1)+g_1'(0)-g_2'(1)= \alpha g_1(0)\,,
$$
 % -------------- %
hence using \eqref{eq:4sten02} we can rewrite it in the form
 % -------------- %
\begin{equation}
  f_1(1) = g_1(0) = \omega_j g_1(1)\,,\quad -f_1'(1)+g_1'(0)-\omega_j g_1'(1) = \alpha g_1(0)\,,\quad j=0,1,2\,. \label{eq:4sten:bcb}
\end{equation}
 % -------------- %
We infer that the Hamiltonian on the tetrahedron is unitarily equivalent to the direct sum of three operators on the graph drawn in Figure~\ref{fig2}. In the first case, the boundary condition at the vertex $A$ is Robin \eqref{eq:4sten:bc:1a} and at the vertex $B$ we have the condition \eqref{eq:4sten:bcb} with $\omega_0 = 1$. In the second and third case, there is Dirichlet boundary condition \eqref{eq:4sten:bc23a} at $A$ and the coupling condition \eqref{eq:4sten:bcb} at $B$ with $\omega_1 = \mathrm{e}^{\frac{2\pi i}{3}}$ and $\omega_2 = \mathrm{e}^{\frac{4\pi i}{3}}$, respectively.

 % -------------- %
\begin{figure}
\centering
\includegraphics[height=3cm]{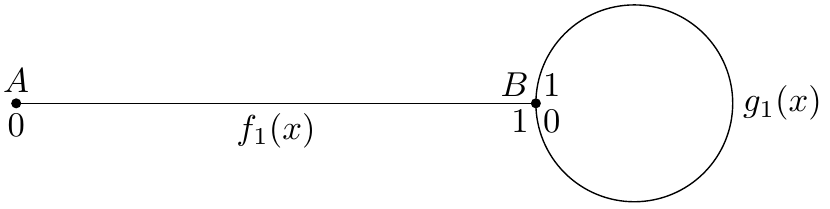}
\caption{The graph supporting component operators for the tetrahedron}
\label{fig2}
\end{figure}
 % -------------- %

Let us start by investigating the first named case. We use the Ansatz
 % -------------- %
$$
  f_1(x) = a\sin{kx}+b\cos{kx}\,,\quad g_1(x) = c\sin{kx}+d\cos{kx}\,,
$$
 % -------------- %
leading to the set of equations
 % -------------- %
\begin{eqnarray*}
ka= b\frac{\alpha}{3}\,,\\
a\sin{k}+b\cos{k} = d = c\sin{k}+d\cos{k}\,,\\
-ak\cos{k}+bk\sin{k}+ck-ck\cos{k}+dk\sin{k}=\alpha d\,,
\end{eqnarray*}
 % -------------- %
which can be rewritten as
 % -------------- %
\begin{eqnarray*}
  b\frac{\alpha}{3}\sin{k}+bk\cos{k}-dk = 0\,,\\
  c\sin{k}+d\cos{k}-d=0\,,\\
  -b\frac{\alpha}{3}\cos{k} + bk\sin{k} +ck(1-\cos{k})+d(k\sin{k}-\alpha) = 0\,.
\end{eqnarray*}
 % -------------- %
The condition of solvability of this system,
 % -------------- %
$$
  k^2 (2\cos{k}+3\sin^2{k}-2)+\alpha k \sin{k} \left(\frac{2}{3}-2\cos{k}\right) -\frac{1}{3}\alpha^2 \sin^2{k} = 0\,,
$$
 % -------------- %
can be cast into the form
 % -------------- %
$$
  \sin{\frac{k}{2}}\left[k^2\sin{\frac{k}{2}}\left(-1+3\cos^2{\frac{k}{2}}\right)+\frac{1}{3}\alpha k \cos{\frac{k}{2}}(1-3\cos{k})-\frac{1}{3}\alpha^2\sin{\frac{k}{2}}\cos^2{\frac{k}{2}}\right]
$$
 % -------------- %
The first factor gives the square roots of eigenvalues in the $k$-plane equal to $k = 2n\pi,\; n\in\mathbb{N}$. Regarding the square bracket as a polynomial in $k$ we note that the coefficients of the first and zeroth power are bounded, hence the remaining (square roots of) eigenvalues approach the numbers $k = 2\arccos{\left(\pm\frac{1}{\sqrt{3}}\right)}$ as $k\to \infty$; we observe that these values are exact in the Kirchhoff case, $\alpha=0$.

 % -------------- %
\begin{remark}
{\rm For the sake of brevity we will in what follows mostly speak of the values of $k$ as of eigenvalues having in mind that they are in reality the square roots of those. In this connection recall that the spectrum of our graph Hamiltonians need not be positive -- this is the case for the preferred-orientation coupling \cite{ETa} and for the $\delta$-coupling if $\alpha<0$, however, in this paper we are concerned with the high-energy asymptotic behavior only so we have always $k>0$.}
\end{remark}
 % -------------- %

In the second and third case, we obtain after using the Ansatz
 % -------------- %
$$
  f_1(x) = a\sin{kx}\,,\quad g_1(x) = c\sin{kx}+d\cos{kx}\,,
$$
 % -------------- %
the equations
 % -------------- %
\begin{eqnarray*}
  a\sin{k} = d = \omega_j (c\sin{k}+d\cos{k})\,,\\
  -ak \cos{k} + ck -\omega_j (ck \cos{k}-dk\sin{k}) = \alpha d\,.
\end{eqnarray*}
 % -------------- %
The solvability condition of this system, i.e. the secular equation, is
 % -------------- %
$$
  \sin{k}[k(\omega_j^2-3\omega_j\cos{k}+1)-\alpha \omega_j\sin{k}] = 0\,,
$$
 % -------------- %
which can be using $\omega_j^2+\omega_j + 1 = 0$ rewritten as
 % -------------- %
$$
\sin{k}\left[k(1+3\cos{k})+\alpha\sin{k}\right] = 0\,.
$$
 % -------------- %
One group of eigenvalues consists of $k = n\pi$, $n\in \mathbb{Z}$, the elements of the second one approach $k=\arccos{(-\frac{1}{3})}$ as $k\to \infty$, again coinciding with them for $\alpha=0$.

%%%%%%%%%%%%%%%
\subsection{Preferred-orientation coupling condition}
%%%%%%%%%%%%%%%

Let us next consider the vertex coupling which prefers one orientation as proposed in \cite{ETa}. In that case, at the point $A$ we have
 % -------------- %
$$
  \begin{pmatrix}-1 & 1 & 0\\ 0 & -1 & 1\\ 1 & 0 & -1\end{pmatrix}\begin{pmatrix}f_1(0)\\ f_2(0)\\ f_3(0)\end{pmatrix} + i  \begin{pmatrix}1 & 1 & 0\\ 0 & 1 & 1\\ 1 & 0 & 1\end{pmatrix}\begin{pmatrix}f_1'(0)\\ f_2'(0)\\ f_3'(0)\end{pmatrix} = 0\,.
$$
 % -------------- %
To be explicit, the orientation is chosen in the following way: if seen from the exterior of the tetrahedron, the entries of the vector of function values are arranged in the clockwise direction. For $j = 0$, this equation simplifies due to \eqref{eq:4sten01} to
 % -------------- %
\begin{equation}
  -f_1(0)+f_1(0)+i(f_1'(0)+f_1'(0)) = 0\,,\quad \Rightarrow \quad f_1'(0) = 0\,.\label{eq:4sten:et:0a}
\end{equation}
 % -------------- %
For $j =1$ and $j = 2$ we have
 % -------------- %
\begin{equation}
  -f_1(0)+\omega_j f_1(0) + i(f_1'(0)+\omega_j f_1'(0)) = 0\,,\quad \Rightarrow \quad f_1'(0) = i \frac{-1+\omega_j}{\omega_j+1} f_1(0) = (-1)^j\sqrt{3}f_1(0)\,.\label{eq:4sten:et:12a}
\end{equation}
 % -------------- %
At the vertex $B$ we have the coupling condition
 % -------------- %
$$
  \begin{pmatrix}-1 & 1 & 0\\ 0 & -1 & 1\\ 1 & 0 & -1\end{pmatrix}\begin{pmatrix}f_1(1)\\ g_1(0)\\ g_2(1)\end{pmatrix} + i  \begin{pmatrix}1 & 1 & 0\\ 0 & 1 & 1\\ 1 & 0 & 1\end{pmatrix}\begin{pmatrix}-f_1'(1)\\ g_1'(0)\\ -g_2'(1)\end{pmatrix}\,.
$$
 % -------------- %
Using \eqref{eq:4sten02} we get
 % -------------- %
\begin{equation}
  \begin{pmatrix}-1 & 1 & 0\\ 0 & -1 & 1\\ 1 & 0 & -1\end{pmatrix}\begin{pmatrix}f_1(1)\\ g_1(0)\\ \omega_j g_1(1)\end{pmatrix} + i  \begin{pmatrix}1 & 1 & 0\\ 0 & 1 & 1\\ 1 & 0 & 1\end{pmatrix}\begin{pmatrix}-f_1'(1)\\ g_1'(0)\\ -\omega_j g_1'(1)\end{pmatrix}\,. \label{eq:4sten:et:b}
\end{equation}
 % -------------- %
Again, the Hamiltonian on the tetrahedron is unitarily equivalent to direct sum of three operators on the graph depicted in Figure~\ref{fig2}, the first one has Neumann boundary condition \eqref{eq:4sten:et:0a} at $A$ and the condition~\eqref{eq:4sten:et:b} with $\omega_0 = 1$ at $B$. The other two have Robin boundary condition \eqref{eq:4sten:et:12a} at $A$ and the coupling \eqref{eq:4sten:et:b} at $B$, both with $\omega_1 = \mathrm{e}^{\frac{2\pi i}{3}}$ and  $\omega_2 = \mathrm{e}^{\frac{4\pi i}{3}}$, respectively.

To find the secular equation for the first partial graph, $j=0$, we use the Ansatz
 % -------------- %
$$
  f_1(x) = b\cos{kx}\,,\quad g_1(x) = c\sin{kx}+d\cos{kx}\,.
$$
 % -------------- %
The coupling conditions then yield
 % -------------- %
\begin{eqnarray*}
  -b\cos{k}+d+ik b \sin{k}+ ikc = 0\,,\\
  -d+c\sin{k}+d\cos{k}+ikc-ikc\cos{k}+ikd \sin{k} = 0\,,\\
  b\cos{k}-c\sin{k}-d\cos{k}+ikb \sin{k}-ikc \cos{k}+ikd\sin{k} = 0\,,
\end{eqnarray*}
 % -------------- %
and the secular equation is
 % -------------- %
$$
  -2ik (k^2\sin^2{k}+3\sin^2{k}+2\cos{k}-2) = 0\,,
$$
 % -------------- %
which can be rewritten as
 % -------------- %
\begin{equation}
  k\sin^2{\frac{k}{2}} \left(k^2\cos^2{\frac{k}{2}}+3\cos^2{\frac{k}{2}}-1\right) = 0\,.\label{eq:4sten:et:1g}
\end{equation}
 % -------------- %
One group of eigenvalues consists of $k = 2n\pi, \; n\in\mathbb{N}$, the elements of other approach the numbers $k = \pi +2n\pi$ as $k\to \infty$.

For the second and the third operator we use the Ansatz
 % -------------- %
$$
  f_1(x) = a\sin{kx} + b\cos{kx}\,,\quad g_1(x) = c\sin{kx}+d\cos{kx}\,.
$$
 % -------------- %
The coupling conditions then become
 % -------------- %
\begin{eqnarray*}
  ka = (-1)^j\sqrt{3} b\,,\\
  -a\sin{k}-b\cos{k}+d-ika\cos{k}+ikb\sin{k}+ikc = 0\,,\\
  -d+\omega_j c \sin{k}+\omega_j d \cos{k}+ ikc -i\omega_j k c\cos{k}+i\omega_j k d\sin{k} = 0\,,\\
  a\sin{k} + b\cos{k}-\omega_j c\sin{k}-\omega_j d\cos{k}-ika \cos{k}+ikb\sin{k}-i\omega_j ck \cos{k}+i \omega_j kd \sin{k} = 0\,,
\end{eqnarray*}
 % -------------- %
and the secular equation is
 % -------------- %
\begin{multline*}
  -2i k^4\omega_j \sin^2{k} + k^3[2i\sqrt{3}(-1)^j\omega_j \cos{k}\sin{k}-2\omega_j^2\sin{k}+2\sin{k}]+
\\
 + k^2[2\sqrt{3}(-1)^j\omega_j^2\cos{k}-2i\omega_j^2\cos{k}-6i\omega_j\sin^2{k}-2\sqrt{3}(-1)^j\cos{k}+4i\omega_j-2i\cos{k}]+
 \\
  +k [-2i\sqrt{3}(-1)^j\omega_j^2\sin{k}+6i\sqrt{3}(-1)^j\omega_j\cos{k}\sin{k}-2i\sqrt{3}(-1)^j\sin{k}] = 0
\end{multline*}
 % -------------- %
Using
 % -------------- %
$$
  \omega_j^2+1 = -\omega_j\,,\quad \omega_j^2-1 = \omega_j i \sqrt{3} (-1)^{j+1}
$$
 % -------------- %
we obtain
 % -------------- %
\begin{equation}
  k\sin{\frac{k}{2}}\left[k\sin{\frac{k}{2}}-(-1)^j\sqrt{3}\cos{\frac{k}{2}}\right]\left[k^2\cos^2{\frac{k}{2}}+3\cos^2{\frac{k}{2}}-1\right] = 0\,.\label{eq:4sten:et:23g}
\end{equation}
 % -------------- %
The first set of eigenvalues is $k = 2n\pi$, $n\in\mathbb{Z}$, the elements of the second approach the same values as $k\to \infty$, while the elements of the third approach $k = \pi+2n\pi$, $n\in\mathbb{Z}$, as $k \to \infty$.

We can also asses the convergence rate as stated in the following theorem:

 % -------------- %
\begin{theorem}
For the tetrahedron with the preferred-orientation coupling the (square roots of) eigenvalues $k$ are for large $k$ contained in the intervals
 % -------------- %
$$
  k \in \left(n\pi -\frac{2\sqrt{3}}{k}+\mathcal{O}\left(\frac{1}{k^2}\right),n\pi+\frac{2\sqrt{3}}{k}+\mathcal{O}\left(\frac{1}{k^2}\right)\right)\,,\quad n\in\mathbb{Z}\,.
$$
 % -------------- %
\end{theorem}
 % -------------- %
\begin{proof}
For the first component operator, we have from \eqref{eq:4sten:et:1g}
 % -------------- %
$$
  \left|\cos{\frac{k}{2}}\right|\leq\frac{\sqrt{2}}{k}\quad \Rightarrow \quad |k-\pi-2n\pi|\leq \frac{2\sqrt{2}}{k} + \mathcal{O}\left(\frac{1}{k^2}\right)
$$
 % -------------- %
For the second and third one, we have either
 % -------------- %
$$
  \left|\sin{\frac{k}{2}}\right|\leq\frac{\sqrt{3}}{k}\quad \Rightarrow \quad |k-2n\pi|\leq \frac{2\sqrt{3}}{k} + \mathcal{O}\left(\frac{1}{k^2}\right)
$$
 % -------------- %
or
 % -------------- %
$$
  \left|\cos{\frac{k}{2}}\right|\leq\frac{\sqrt{2}}{k}\quad \Rightarrow \quad |k-\pi-2n\pi|\leq \frac{2\sqrt{2}}{k} + \mathcal{O}\left(\frac{1}{k^2}\right)
$$
 % -------------- %
which proves the claim.
\end{proof}

We see that in this case the high eigenvalues coincide with an $\mathcal{O}(n^{-1})$ error with those of the Dirichlet Laplacian on a unit-length interval.

%%%%%%%%%%%%%%%
\section{The cube}
%%%%%%%%%%%%%%%

 % -------------- %
\begin{figure}
\centering
\includegraphics[height=6cm]{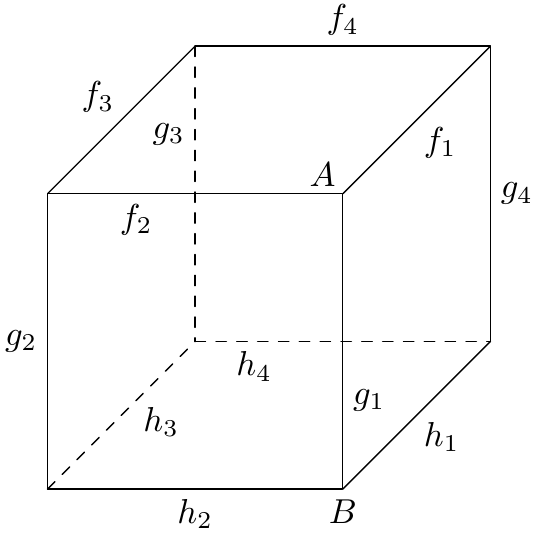}
\caption{A cube}
\label{fig3}
\end{figure}
 % -------------- %
Consider next a cube with unit length edges supporting a wavefunction with components described in Figure~\ref{fig3}. We employ the rotation operator $T_4$, this describing rotations by $\pi/4$ around the axis going through the centers of the upper and lower face. In the eigensubspaces of $T_4$ it holds
 % -------------- %
\begin{eqnarray}
  f_2(x) = \omega_j f_1(x)\,,\quad f_3(x) = \omega_j^2 f_1(x)\,,\quad f_4(x) = \omega_j^3 f_1(x)\,,\label{eq:cube:01}\\
  g_2(x) = \omega_j g_1(x)\,,\quad g_3(x) = \omega_j^2 g_1(x)\,,\quad g_4(x) = \omega_j^3 g_1(x)\,,\label{eq:cube:02}\\
  h_2(x) = \omega_j h_1(x)\,,\quad h_3(x) = \omega_j^2 h_1(x)\,,\quad h_4(x) = \omega_j^3 h_1(x)\label{eq:cube:03}
\end{eqnarray}
 % -------------- %
with $\omega_j = \mathrm{e}^{\frac{i\pi}{2}j}$, $j = 0, 1,2,3$. Using these subspaces we can find the direct sum decomposition of our operators is analogy with what we did above for the tetrahedron.

%%%%%%%%%%%%%%%
\subsection{$\delta$-condition}
%%%%%%%%%%%%%%%

We begin again with the $\delta$-coupling imposed at all the vertices. We use the parametrization of the edges by intervals $(0,1)$ such that for the vertex $A$ the corresponding values of $x$ are: $x=0$ for $f_1$, $x= 1$ for $f_2$ and $x=0$ for $g_1$; for the vertex $B$ we have $x=1$ for $g_1$, $x=0$ for $h_1$ and $x=1$ for $h_2$. The coupling condition at $A$ is
 % -------------- %
$$
  f_1(0) = g_1(0) = f_2(1)\,,\quad f_1'(0)+g_1'(0)-f_2'(1)=\alpha f_1(0)\,.
$$
 % -------------- %
which can be using \eqref{eq:cube:01} rewritten as
 % -------------- %
\begin{equation}
  f_1(0) = g_1(0) = \omega_j f_1(1)\,,\quad f_1'(0)+g_1'(0)-\omega_j f_1'(1)=\alpha f_1(0)\,.\label{eq:cube:delta:a}
\end{equation}
 % -------------- %
At the vertex $B$ we have
 % -------------- %
$$
  g_1(1) = h_1(0)= h_2(1)\,,\quad -g_1'(1)+h_1'(0)-h_2'(1) = \alpha h_1(0)\,.
$$
From \eqref{eq:cube:03} it follows
\begin{equation}
  g_1(1) = h_1(0)= \omega_j h_1(1)\,,\quad -g_1'(1)+h_1'(0)-\omega_j h_1'(1) = \alpha h_1(0)\,.\label{eq:cube:delta:b}
\end{equation}
 % -------------- %
We obtain four component operators supported by the graph showed in Figure~\ref{fig4} with the coupling conditions \eqref{eq:cube:delta:a} and \eqref{eq:cube:delta:b} with $\omega_j = \mathrm{e}^{\frac{i
\pi}{2}j}$.

 % -------------- %
\begin{figure}
\centering
\includegraphics[height=2.5cm]{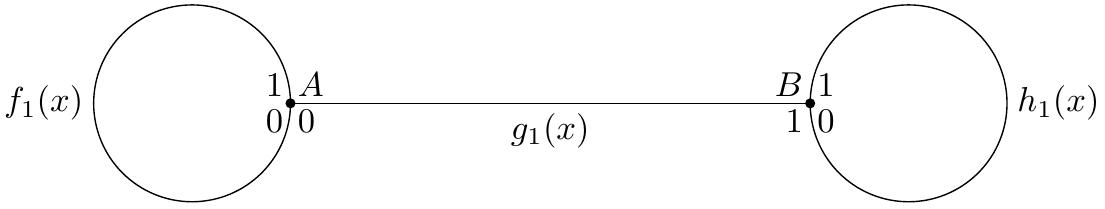}
\caption{The graph supporting component operators for the cube}
\label{fig4}
\end{figure}
 % -------------- %

To obtain the secular equation, we use the Ansatz
 % -------------- %
$$
  f_1(x) = a\sin{kx}+b\cos{kx}\,,\quad g_1(x) = c\sin{kx}+d\cos{kx}\,,\quad h_1(x) = e\sin{kx}+f\cos{kx}\,.
$$
 % -------------- %
Using \eqref{eq:cube:delta:a} and \eqref{eq:cube:delta:b} we obtain the set of equations
 % -------------- %
\begin{eqnarray*}
  b =d=\omega_j a\sin{k}+\omega_j b\cos{k}\,,\\
  ka+kc -\omega_j ka \cos{k}+\omega_j kb\sin{k} = \alpha b\,,\\
  c\sin{k}+d\cos{k} = f = \omega_je\sin{k}+\omega_j f \cos{k}\,,\\
  -ck\cos{k}+dk\sin{k}+ke-\omega_j ke \cos{k} +f\omega_j k \sin{k} = \alpha f\,.
\end{eqnarray*}
 % -------------- %
The secular equation is
 % -------------- %
\begin{multline*}
  \left[\sin{\frac{k}{2}}(k\omega_j^2+k-2k\omega_j\cos{k}-\alpha \omega_j \sin{k})-\omega_j k \sin{k}\cos{\frac{k}{2}}\right]\cdot
\\
\cdot
\left[\cos{\frac{k}{2}}(k\omega_j^2+k-2k\omega_j\cos{k}-\alpha\omega_j\sin{k})+\omega_j k \sin{k}\sin{\frac{k}{2}}\right] = 0\,,
\end{multline*}
 % -------------- %
where the first factor can be rewritten as
 % -------------- %
$$
  k\sin{\frac{k}{2}}\left[\omega_j^2+1-2\omega_j\left(3\cos^2{\frac{k}{2}}-1\right)\right]-\alpha\omega_j \sin{\frac{k}{2}}\sin{k} = 0\,,
$$
 % -------------- %
which leads to the eigenvalues $k = 2n\pi$, $n\in\mathbb{Z}$ and another series the elements of which approach for any nonzero $\alpha\in\mathbb{R}$ and large $k$ the solutions of the Kirchhoff case which can be determined analytically. Specifically, for $\omega_j^2 = -1$ we get $k = 2\arccos{\left(\pm \frac{\sqrt{3}}{3}\right)}$, i.e. $k\approx \pm 1.911+4n\pi$ or $k \approx \pm 4.373+4n\pi$, $n\in\mathbb{Z}$. Similarly for $\omega_j = 1$ we get $k = 2\arccos{\left(\pm \frac{\sqrt{6}}{3}\right)}$, i.e. $k\approx \pm 1.231+4n\pi$ or $k\approx \pm 5.052 + 4n\pi$, $n\in\mathbb{Z}$, and finally, for $\omega_j = -1$ we have $k = \pi +2n\pi$, $n\in\mathbb{Z}$.

The second factor at the left-hand side of the secular equation can be written as
 % -------------- %
$$
  k \cos{\frac{k}{2}}\left[\omega_j^2+1+2\omega_j\left(3\sin^2{\frac{k}{2}}-1\right)\right]-\alpha\omega_j\cos{\frac{k}{2}}\sin{k} = 0\,,
$$
 % -------------- %
which leads to eigenvalues with $k = \pi +2n\pi$ and the others that accumulate at the following values: for $\omega_j^2 = -1$ it is $k = 2\arcsin{\left(\pm \frac{\sqrt{3}}{3}\right)}$, i.e. $k\approx \pm 1.231 +2n\pi$, $n\in\mathbb{Z}$, for $\omega_j = -1$ we have $k = 2\arcsin{\left(\pm \frac{\sqrt{6}}{3}\right)}$, i.e. $k\approx \pm 1.911+2n\pi$, $n\in\mathbb{Z}$, and finally, for $\omega_j = 1$ we have $k = 2n\pi$, $n\in\mathbb{Z}$.

This description can be simplified by comparing the above indicated values modulo $\pi$: we have the series of $k = n\pi$, $n\in \mathbb{Z}$, with multiplicity five and another two that approach the values $k\approx \pm 1.231 + n\pi$ (here the multiplicity is three). We also remark that in the cube case one can obtain this result alternatively using the mirror symmetries and examining the three-prong stars with the $\delta$-coupling at the vertex and the Dirichlet/Neumann conditions at the edges endpoints.

%%%%%%%%%%%%%%%
\subsection{Preferred-orientation coupling}
%%%%%%%%%%%%%%%

Using the same wavefunction notation as above, assuming the rotation is in the clockwise direction when viewed from the exterior of the cube. At the vertex $A$ we have
 % -------------- %
$$
  \begin{pmatrix}-1 & 1 &0 \\ 0 & -1 & 1\\ 1 & 0 & -1\end{pmatrix}\begin{pmatrix}f_1(0) \\ g_1(0) \\ f_2(1)\end{pmatrix} + i \begin{pmatrix}1 & 1 &0 \\ 0 & 1 & 1\\ 1 & 0 & 1\end{pmatrix}\begin{pmatrix}f_1'(0) \\ g_1'(0) \\ -f_2'(1)\end{pmatrix} =0 \,.
$$
 % -------------- %
From \eqref{eq:cube:01} we get
 % -------------- %
\begin{equation}
  \begin{pmatrix}-1 & 1 &0 \\ 0 & -1 & 1\\ 1 & 0 & -1\end{pmatrix}\begin{pmatrix}f_1(0) \\ g_1(0) \\ \omega_j f_1(1)\end{pmatrix} + i \begin{pmatrix}1 & 1 &0 \\ 0 & 1 & 1\\ 1 & 0 & 1\end{pmatrix}\begin{pmatrix}f_1'(0) \\ g_1'(0) \\ -\omega_j f_1'(1)\end{pmatrix} =0 \,.\label{eq:cube:et:a}
\end{equation}
 % -------------- %
Furthermore, at the vertex $B$ we have the condition
 % -------------- %
$$
  \begin{pmatrix}-1 & 1 &0 \\ 0 & -1 & 1\\ 1 & 0 & -1\end{pmatrix}\begin{pmatrix} g_1(1) \\ h_1(0)\\ h_2(1)\end{pmatrix} + i \begin{pmatrix}1 & 1 &0 \\ 0 & 1 & 1\\ 1 & 0 & 1\end{pmatrix}\begin{pmatrix}-g_1'(1) \\ h_1'(0)\\ -h_2'(1)\end{pmatrix} =0 \,.
$$
 % -------------- %
Using \eqref{eq:cube:03} we obtain
 % -------------- %
\begin{equation}
  \begin{pmatrix}-1 & 1 &0 \\ 0 & -1 & 1\\ 1 & 0 & -1\end{pmatrix}\begin{pmatrix} g_1(1) \\ h_1(0)\\ \omega_j h_1(1)\end{pmatrix} + i \begin{pmatrix}1 & 1 &0 \\ 0 & 1 & 1\\ 1 & 0 & 1\end{pmatrix}\begin{pmatrix}-g_1'(1) \\ h_1'(0)\\ -\omega_j h_1'(1)\end{pmatrix} =0 \,.\label{eq:cube:et:b}
\end{equation}
 % -------------- %
There are four component operators on the graph of Figure~\ref{fig4} with the coupling conditions \eqref{eq:cube:et:a} and \eqref{eq:cube:et:b}. We again use the Ansatz
 % -------------- %
$$
  f_1(x) = a\sin{kx}+b\cos{kx}\,,\quad g_1(x) = c\sin{kx}+d\cos{kx}\,,\quad h_1(x) = e\sin{kx}+f\cos{kx}
$$
 % -------------- %
which leads to the equations
 % -------------- %
\begin{eqnarray*}
  -b+d+ika+ikc=0\,,\\
  -d+\omega_j a \sin{k}+\omega_j b \cos{k}+ikc-i\omega_j k a \cos{k}+i \omega_j k b \sin{k} = 0\,,\\
  b-\omega_j a \sin{k} -\omega_j b \cos{k}+ika -i \omega_j ka \cos{k}+i\omega_j kb \sin{k} = 0\,,\\
  -c\sin{k}-d\cos{k}+f-ikc \cos{k}+ikd\sin{k}+ike = 0\,,\\
  -f+\omega_j e \sin{k}+\omega_j f \cos{k}+ike -ik\omega_j e \cos{k}+ik\omega_j f \sin{k}=0\,,\\
  c\sin{k}+d\cos{k}-\omega_j e\sin{k}-\omega_j f \cos{k}-ik c \cos{k}+ikd \sin{k}-ik\omega_j e\cos{k}+ik \omega_j f \sin{k}=0\,.
\end{eqnarray*}
 % -------------- %
The secular equation is then
 % -------------- %
\begin{multline*}
  k^6 \omega_j^2\sin^3{k} + k^4 \sin{k} (\omega_j^4+2\omega_j^3\cos{k}-6\omega_j^2\cos^2{k}+2\omega_j \cos{k}+1) +
\\
 +k^2 \sin{k}(-\omega_j^4+6\omega_j^3\cos{k}-9\omega_j^2\cos^2{k}-\omega_j^2+6\omega_j\cos{k}-1) = 0
\end{multline*}
 % -------------- %
and we arrive at the following theorem:

 % -------------- %
\begin{theorem}
For the cube with the preferred-orientation coupling the (square roots of) eigenvalues $k$ are for large $k$ in the intervals
 % -------------- %
$$
  k \in \left(n\pi -\frac{2\sqrt{3}}{k}+\mathcal{O}\left(\frac{1}{k^2}\right),n\pi+\frac{2\sqrt{3}}{k}+\mathcal{O}\left(\frac{1}{k^2}\right)\right)\,,\quad n\in\mathbb{Z}\,.
$$
 % -------------- %
\end{theorem}
 % -------------- %
\begin{proof}
First, there are eigenvalues given by $\sin{k} = 0$ which belong to the above intervals. The remaining ones are the roots of the equation
 % -------------- %
$$
  \sin^2{k}  + \frac{1}{k^2} (\omega_j^2+2\omega_j\cos{k}-6\cos^2{k}+2\omega_j^{-1}\cos{k}+\omega_j^{-2})+\mathcal{O}\left(\frac{1}{k^4}\right) = 0\,.
$$
 % -------------- %
The modulus of the parenthesis expression does not exceed $12$, from which we infer that $|\sin{k}|\leq \frac{2\sqrt{3}}{k} +\mathcal{O}\left(\frac{1}{k^2}\right)$, and the claim of the theorem follows.
\end{proof}

As in the tetrahedron case, the spectrum accumulates at high energies around the eigenvalues of the Dirichlet Laplacian on a unit-length interval.

%%%%%%%%%%%%%%%
\section{The octahedron}
%%%%%%%%%%%%%%%

 % -------------- %
\begin{figure}
\centering
\includegraphics[height=8cm]{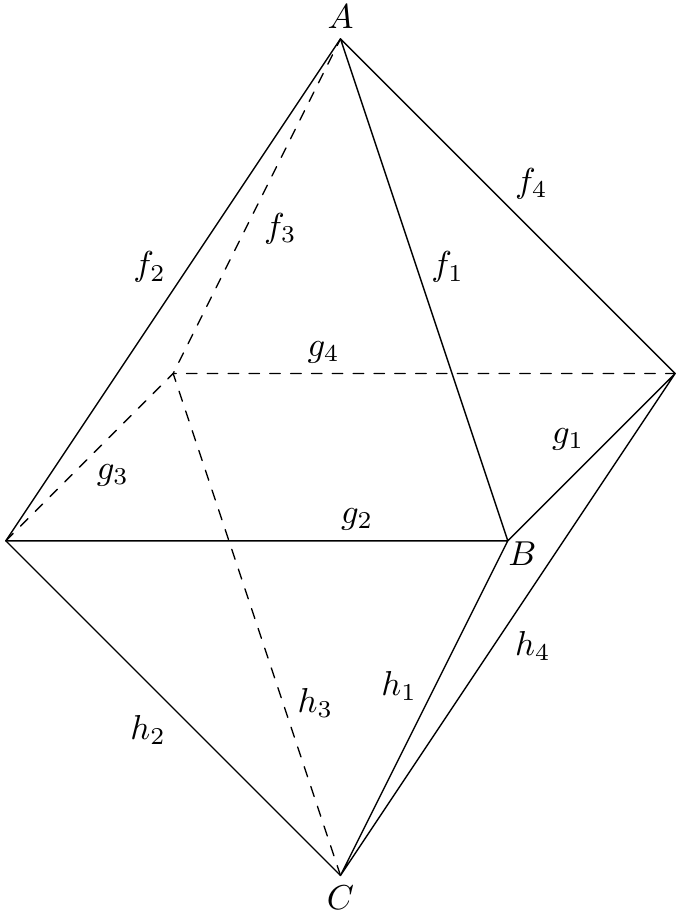}
\caption{An octahedron}
\label{fig5}
\end{figure}
 % -------------- %

Now we consider an octahedron with the wavefunction components labeled as in Figure~\ref{fig5}. There is a 4-fold axis going through vertices $A$ and $C$. The operator of rotation by $\pi/2$ clockwise, as seen from the exterior, of the tetrahedron at the vertex $A$ is $T_4$. In the eigensubspaces of $T_4$ it again holds that
 % -------------- %
\begin{eqnarray}
  f_2(x) = \omega_j f_1(x)\,,\quad f_3(x) = \omega_j^2 f_1(x)\,,\quad f_4(x) = \omega_j^3 f_1(x)\,,\label{eq:8sten:01}\\
  g_2(x) = \omega_j g_1(x)\,,\quad g_3(x) = \omega_j^2 g_1(x)\,,\quad g_4(x) = \omega_j^3 g_1(x)\,,\label{eq:8sten:02}\\
  h_2(x) = \omega_j h_1(x)\,,\quad h_3(x) = \omega_j^2 h_1(x)\,,\quad h_4(x) = \omega_j^3 h_1(x)\label{eq:8sten:03}
\end{eqnarray}
with $\omega_j = \mathrm{e}^{\frac{i\pi}{2}j}$, $j = 0, 1, 2, 3$.
 % -------------- %

%%%%%%%%%%%%%%%
\subsection{$\delta$-condition}
%%%%%%%%%%%%%%%
As before, we begin with the $\delta$-condition of the same strength $\alpha$ at all the vertices. At the vertex $A$ (corresponding to $x=0$ for $f_1$, $f_2$, $f_3$, and $f_4$) we have
 % -------------- %
$$
  f_1(0)=f_2(0) = f_3(0) = f_4(0)\,,\quad f_1'(0) + f_2'(0) + f_3'(0) + f_4'(0)  = \alpha f_1(0)\,.
$$
 % -------------- %
Using \eqref{eq:8sten:01} we obtain for $j=0$ the Robin condition
 % -------------- %
\begin{equation}
  f_1'(0)=\frac{\alpha}{4}f_1(0)\label{eq:8sten:delta:robina}
\end{equation}
 % -------------- %
and for $j=1,2,3$ the Dirichlet condition
 % -------------- %
\begin{equation}
  f_1(0)=0\,.\label{eq:8sten:delta:dirichleta}
\end{equation}
 % -------------- %
The situation at the vertex $C$ (referring to $x=1$ for $h_1$, $h_2$, $h_3$, and $h_4$) is similar. The coupling condition is
 % -------------- %
$$
  h_1(1)=h_2(1) = h_3(1)=h_4(1)\,,\quad -h_1'(1) -h_2'(1) -h_3'(1) -h_4'(1)  = \alpha h_1(1)\,.
$$
 % -------------- %
Using equation~\eqref{eq:8sten:03} we obtain for $j = 0$ the Robin condition
 % -------------- %
\begin{equation}
  -h_1'(1)=\frac{\alpha}{4}h_1(1) \label{eq:8sten:delta:robinc}
\end{equation}
 % -------------- %
and for $j = 1, 2, 3$ the Dirichlet condition
 % -------------- %
\begin{equation}
  h_1(0)=0\,.\label{eq:8sten:delta:dirichletc}
\end{equation}
 % -------------- %
At the vertex $B$ (corresponding to $x=1$ for $f_1$ and $g_2$, and to $x=0$ for $g_2$ and $h_1$) the coupling condition is
 % -------------- %
$$
  f_1(1) = g_1(0) = g_2(1) = h_1(0)\,,\quad -f_1'(1)+g_1'(0)-g_2'(1) + h_1'(0) = \alpha g_1(0)\,.
$$
 % -------------- %
Using now equation~\eqref{eq:8sten:02} we obtain
 % -------------- %
\begin{equation}
  f_1(1) = g_1(0) = \omega_j g_1(1) = h_1(0)\,,\quad -f_1'(1)+g_1'(0)-\omega_j g_1'(1) + h_1'(0) = \alpha g_1(0)\,.\label{eq:8sten:delta:b}
\end{equation}
 % -------------- %
Hence we have four component operators on the graph shown in Figure~\ref{fig6}, first one with Robin conditions~\eqref{eq:8sten:delta:robina} and \eqref{eq:8sten:delta:robinc} at the vertices $A$ and $C$ and the $\delta$-condition \eqref{eq:8sten:delta:b} with $\omega_0 = 1$ at $B$. The other three operators have Dirichlet condition at $A$ and $C$ (by equations~\eqref{eq:8sten:delta:dirichleta} and \eqref{eq:8sten:delta:dirichletc}) and the condition \eqref{eq:8sten:delta:b} at $B$ with $\omega_j = \mathrm{e}^{\frac{i\pi}{2}j}$, $j = 1, 2, 3$.

 % -------------- %
\begin{figure}
\centering
\includegraphics[width=0.8\textwidth]{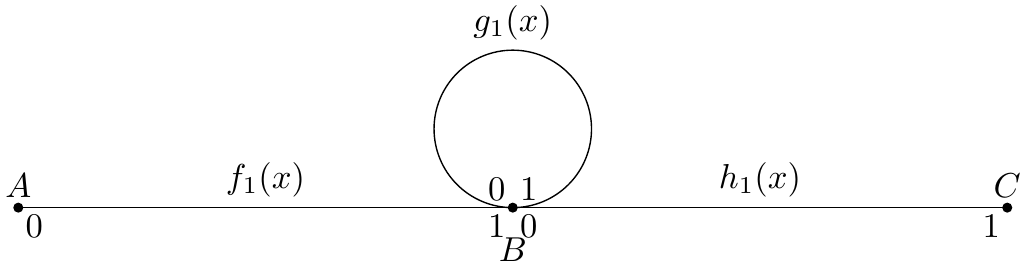}
\caption{The graph supporting component operators for the octahedron}
\label{fig6}
\end{figure}
 % -------------- %

First we address the situation with $j = 0$. We use the Ansatz
 % -------------- %
$$
  f_1(x) = a\sin{kx}+b\cos{kx}\,,\quad g_1(x) = c\sin{kx}+d\cos{kx}\,,\quad h_1(x) = e\sin{kx}+f\cos{kx}
$$
 % -------------- %
and obtain the equations
 % -------------- %
\begin{eqnarray*}
  ka = \frac{\alpha}{4} b\,,\\
  a\sin{k}+b\cos{k} = d = c\sin{k}+ d\cos{k} = f \,,\\
  -ak\cos{k}+bk \sin{k}+ck -ck\cos{k}+dk\sin{k}+ek = \alpha d\,,\\
  -ek\cos{k}+fk\sin{k} = \frac{\alpha}{4}e\sin{k}+\frac{\alpha}{4} f \cos{k}\,.
\end{eqnarray*}
 % -------------- %
The corresponding secular equation is
 % -------------- %
\begin{multline*}
 k^3(64\cos{k}\sin^2{k}-32\sin^2{k}+32\cos{k}+32)+k^2\alpha\sin{k}(48\sin^2{k}+16\cos{k}-40)+
\\
 +k\alpha^2\sin^2{k}(-12\cos{k}+2)-\alpha^3\sin^{3}k = 0\,.
\end{multline*}
 % -------------- %
It leads to the eigenvalues which approach $k = \frac{\pi}{2}+n\pi$ and $k = \pi+2n\pi$, $n\in \mathbb{Z}$.

Passing to $j = 1, 2, 3$, we employ the Ansatz
 % -------------- %
$$
  f_1(x) = a\sin{kx}\,,\quad g_1(x) = c\sin{kx}+d\cos{kx}\,,\quad h_1(x) = -e\cos{k}\sin{kx}+e\sin{k}\cos{kx}
$$
 % -------------- %
which leads to the set of equations
\begin{eqnarray*}
  a\sin{k} = d =\omega_jc\sin{k}+\omega_j d \cos{k} = e\sin{k}\,,\\
  -ak\cos{k}+ck-\omega_j ck\cos{k}+\omega_j dk \sin{k}-ek\cos{k} = \alpha d\,.
\end{eqnarray*}
 % -------------- %
The secular equation is in this case
 % -------------- %
$$
  k (-\omega_j^2+4\omega_j \cos{k}-1)\sin^{2}{k}+\alpha \omega_j\sin^{3}{k} = 0\,;
$$
 % -------------- %
from it we obtain the eigenvalues $k = n\pi$, $n\in\mathbb{Z}$, and the others that approach asymptotically $k = \pm \frac{2\pi}{3}+2n\pi$, $n\in\mathbb{Z}$ (for $\omega_j = -1$) and $k = \frac{\pi}{2}+n\pi$, $n\in\mathbb{Z}$ (for $\omega_j^2 = -1$); these values are again the exact solutions in the Kirchhoff case, $\alpha=0$.

%%%%%%%%%%%%%%%
\subsection{Preferred-orientation coupling}
%%%%%%%%%%%%%%%
In that case we have in the vertex $A$
 % -------------- %
$$
  \begin{pmatrix}-1& 1& 0 &0\\ 0 & -1 & 1 & 0\\ 0 & 0 & -1 & 1\\ 1 & 0 & 0 & -1\end{pmatrix} \begin{pmatrix}f_1(0)\\ f_2(0)\\ f_3(0)\\ f_4(0)\end{pmatrix} +
 i \begin{pmatrix}1& 1& 0 &0\\ 0 & 1 & 1 & 0\\ 0 & 0 & 1 & 1\\ 1 & 0 & 0 & 1\end{pmatrix} \begin{pmatrix}f_1'(0)\\ f_2'(0)\\ f_3'(0)\\ f_4'(0)\end{pmatrix} =0\,.
$$
 % -------------- %
Using \eqref{eq:8sten:01} we obtain for $j = 0$ the Neumann boundary condition
 % -------------- %
\begin{equation}
 f_1'(0) = 0\,, \label{eq:8sten:et:a0}
\end{equation}
 % -------------- %
and for $j = 1, 2, 3$ the Robin condition
 % -------------- %
$$
  f_1'(0) = i\frac{\omega_j-1}{\omega_j +1}f_1(0)\,,
$$
 % -------------- %
in particular,
 % -------------- %
\begin{eqnarray}
  f_1'(0) = -f_1(0)\,,\quad j=1\,,\label{eq:8sten:et:a1}\\
  f_1(0) = 0\,,\quad j=2\,,\label{eq:8sten:et:a2}\\
  f_1'(0) = f_1(0)\,,\quad j=3\,.\label{eq:8sten:et:a3}
\end{eqnarray}
 % -------------- %
Similarly, for the vertex $C$ we obtain the coupling condition
 % -------------- %
$$
  \begin{pmatrix}-1& 1& 0 &0\\ 0 & -1 & 1 & 0\\ 0 & 0 & -1 & 1\\ 1 & 0 & 0 & -1\end{pmatrix} \begin{pmatrix}h_4(1)\\ h_3(1)\\ h_2(1)\\ h_1(1)\end{pmatrix} +
 i \begin{pmatrix}1& 1& 0 &0\\ 0 & 1 & 1 & 0\\ 0 & 0 & 1 & 1\\ 1 & 0 & 0 & 1\end{pmatrix}\begin{pmatrix}-h_4'(1)\\ -h_3'(1)\\ -h_2'(1)\\ -h_1'(1)\end{pmatrix}  =0\,.
$$
 % -------------- %
Using \eqref{eq:8sten:03} we find
 % -------------- %
\begin{eqnarray}
  h_1'(1) = 0\,,\quad j=0\,,\label{eq:8sten:et:c0}\\
  h_1'(1) = -h_1(1)\,,\quad j=1\,,\label{eq:8sten:et:c1}\\
  h_1(1) = 0\,,\quad j=2\,,\label{eq:8sten:et:c2}\\
  h_1'(1) = h_1(1)\,,\quad j=3\,.\label{eq:8sten:et:c3}
\end{eqnarray}
 % -------------- %
Furthermore, at the vertex $B$ we have
 % -------------- %
$$
  \begin{pmatrix}-1& 1& 0 &0\\ 0 & -1 & 1 & 0\\ 0 & 0 & -1 & 1\\ 1 & 0 & 0 & -1\end{pmatrix} \begin{pmatrix}f_1(1)\\ g_1(0)\\ h_1(0)\\ g_2(1)\end{pmatrix} +
 i \begin{pmatrix}1& 1& 0 &0\\ 0 & 1 & 1 & 0\\ 0 & 0 & 1 & 1\\ 1 & 0 & 0 & 1\end{pmatrix}\begin{pmatrix}-f_1'(1)\\ g_1'(0)\\ h_1'(0)\\ -g_2'(1)\end{pmatrix}  =0\,,
$$
 % -------------- %
and using \eqref{eq:8sten:02} we get
 % -------------- %
\begin{equation}
  \begin{pmatrix}-1& 1& 0 &0\\ 0 & -1 & 1 & 0\\ 0 & 0 & -1 & 1\\ 1 & 0 & 0 & -1\end{pmatrix} \begin{pmatrix}f_1(1)\\ g_1(0)\\ h_1(0)\\\omega_j g_1(1)\end{pmatrix} +
 i \begin{pmatrix}1& 1& 0 &0\\ 0 & 1 & 1 & 0\\ 0 & 0 & 1 & 1\\ 1 & 0 & 0 & 1\end{pmatrix}\begin{pmatrix}-f_1'(1)\\ g_1'(0)\\ h_1'(0)\\ -\omega_j g_1'(1)\end{pmatrix} = 0\,.\label{eq:8sten:et:b}
\end{equation}
 % -------------- %
In this way we obtain four component operators, the first with Neumann boundary conditions at $A$ and $C$ (by eq.~\eqref{eq:8sten:et:a0} and \eqref{eq:8sten:et:c0}) and the coupling condition \eqref{eq:8sten:et:b} with $\omega_0 = 1$ at $B$. The second has Robin boundary conditions at $A$ and $C$ (eq.~\eqref{eq:8sten:et:a1} and \eqref{eq:8sten:et:c1}) and the coupling condition~\eqref{eq:8sten:et:b} with $\omega_1 = i$ at $B$. The third component opeartor has Dirichlet conditions at $A$ and $C$ (by eq.~\eqref{eq:8sten:et:a2} and \eqref{eq:8sten:et:c2}) and the coupling condition~\eqref{eq:8sten:et:b} with $\omega_2 = -1$ at $B$. Finally, the fourth operator has Robin boundary conditions at $A$ and $C$ (by eq.~\eqref{eq:8sten:et:a3} and \eqref{eq:8sten:et:c3}) and the coupling condition~\eqref{eq:8sten:et:b} with $\omega_3 = -i$ at $B$. Each of these operators is supported by the graph sketched in Figure~\ref{fig6}.

For the first component operator, $j=0$, we use the Ansatz
 % -------------- %
\begin{eqnarray*}
  f_1(x) = b\cos{kx}\,,\\
  g_1(x) = c\sin{kx}+d\cos{kx}\,,\\
  h_1(x) = e(\sin{k}\sin{kx}+\cos{k}\cos{kx})\,.
\end{eqnarray*}
 % -------------- %
The coupling conditions lead to the following set of equations,
 % -------------- %
\begin{eqnarray*}
  -b \cos{k}+d+ikb \sin{k}+ikc =0\,,\\
  -d+e\cos{k}+ikc+ike\sin{k} = 0\,,\\
  -e\cos{k}+c\sin{k}+d\cos{k}+ike\sin{k}-ikc\cos{k}+ikd\sin{k} = 0\,,\\
  b\cos{k}-c\sin{k}-d\cos{k}+ikb\sin{k}-ikc\cos{k}+ikd\sin{k} = 0\,.
\end{eqnarray*}
 % -------------- %
The respective secular equation then is
 % -------------- %
$$
 4ik [k^2(-1-2\cos{k})\sin^2{k}+(1-2\cos{k})\sin^2{k}+\cos{k}-1] = 0\,,
$$
 % -------------- %
which can be rewritten as
 % -------------- %
\begin{equation}
  k\sin^2{\frac{k}{2}}\left(4\cos^2{\frac{k}{2}}-1\right)\left[k^2\cos^2{\frac{k}{2}}+\cos^2{\frac{k}{2}}-\frac{1}{2}\right] = 0\,.\label{eq:8sten:et:seceq}
\end{equation}
 % -------------- %
The eigenvalues coming from $\sin^2{\frac{k}{2}}$ are $k = 2\pi n$, $n\in \mathbb{Z}$, the eigenvalues from $4\cos^2{\frac{k}{2}}-1$ are $k =\pm \frac{2\pi}{3}+4n\pi$ and $k =\pm \frac{4\pi}{3}+4n\pi$, $n\in \mathbb{Z}$, the remaining eigenvalues approach $k = \pi +2n\pi$, $n\in \mathbb{Z}$, as $k\to \infty$.

Passing to $j = 1$, we use the Ansatz
 % -------------- %
\begin{eqnarray*}
  f_1(x) = a(\sin{kx}-k\cos{kx})\,,\\
  g_1(x) = c\sin{kx}+d\cos{kx}\,,\\
  h_1(x) = e(-\cos{k}+k\sin{k})\sin{kx}+e(\sin{k}+k\cos{k})\cos{kx}\,.
\end{eqnarray*}
 % -------------- %
The coupling conditions then lead to
 % -------------- %
\begin{eqnarray*}
  a(-\sin{k}+k\cos{k})+d+ika(-\cos{k}-k\sin{k})+ikc=0\,,\\
  -d+e(\sin{k}+k\cos{k})+ikc+ike(-\cos{k}+k\sin{k}) = 0\,,\\
  -e(\sin{k}+k\cos{k})+ic\sin{k}+id\cos{k}+ike(-\cos{k}+k\sin{k})+ck\cos{k}-dk\sin{k} = 0\,,\\
  a(\sin{k}-k\cos{k})-ic\sin{k}-id\cos{k}-ika(\cos{k}+k\sin{k})+ck\cos{k}-dk\sin{k} = 0\,,
\end{eqnarray*}
 % -------------- %
and consequently, to the secular equation
 % -------------- %
$$
  8k[k^4(-\cos{k}\sin^2{k})+k^2(\cos{k}(1-2\sin^2{k})-1)-\cos{k}\sin^2{k}] = 0\,.
$$
 % -------------- %
which yields
 % -------------- %
$$
  k\sin^2{\frac{k}{2}}\left[2k^4\cos{k}\cos^2{\frac{k}{2}}+k^2\left(1+4\cos^2{\frac{k}{2}}\cos{k}\right)+2\cos{k}\cos^2{\frac{k}{2}}\right] = 0\,.
$$
 % -------------- %
The factor $\sin^2{\frac{k}{2}}$ leads to the eigenvalues $k = 2n\pi$, $n\in\mathbb{Z}$, the other eigenvalues approach $k = \frac{\pi}{2}+n\pi$ and $k = \pi + 2n\pi$, $n \in\mathbb{Z}$.

Similarly, for $j = 2$ we use the Ansatz
 % -------------- %
\begin{eqnarray*}
  f_1(x) = a\sin{kx}\,,\\
  g_1(x) = c\sin{kx}+d\cos{kx}\,,\\
  h_1(x) = e(\sin{k}\cos{kx}-\cos{k}\sin{kx})\,.
\end{eqnarray*}
 % -------------- %
The assumed coupling conditions give the set of equations
 % -------------- %
\begin{eqnarray*}
  -a\sin{k}+d-ika\cos{k}+ikc=0\,,\\
  -d+e\sin{k}+ikc-ike\cos{k} = 0\,,\\
  -e\sin{k}-c\sin{k}-d\cos{k}-ike\cos{k}+ikc\cos{k}-ikd\sin{k} = 0\,,\\
  a\sin{k}+c\sin{k}+d\cos{k}-ika\cos{k}+ikc\cos{k}-ikd\sin{k} = 0\
\end{eqnarray*}
 % -------------- %
for which the secular equation is
 % -------------- %
$$
  4ik^3[\cos{k}-1+\sin^2{k}(1-2\cos{k})]-4ik \sin^2{k}(1+2\cos{k})=0\,,
$$
 % -------------- %
which can be rewritten as
 % -------------- %
$$
  k \sin^2{\frac{k}{2}}\left(\cos^2{\frac{k}{2}-\frac{1}{4}}\right)\left[k^2\left(\cos{\frac{k}{2}}-\frac{\sqrt{2}}{2}\right)\left(\cos{\frac{k}{2}}+\frac{\sqrt{2}}{2}\right)+\cos^2{\frac{k}{2}}\right] = 0\,.
$$
 % -------------- %
The eigenvalues are $k = 2n\pi$, $k = \pm\frac{2\pi}{3}+4n\pi$ and $k = \pm\frac{4\pi}{3}+4n\pi$, $n\in\mathbb{Z}$; the remaining ones approach the points $k = \frac{\pi}{2}+n\pi$, $n\in\mathbb{Z}$, as $k\to \infty$.

Finally, for $j = 3$ the Ansatz is
 % -------------- %
\begin{eqnarray*}
  f_1(x) = a(\sin{kx}+k\cos{kx})\,,\\
  g_1(x) = c\sin{kx}+d\cos{kx}\,,\\
  h_1(x) = e(-\cos{k}-k\sin{k})\sin{kx}+e(\sin{k}-k\cos{k})\cos{kx}\,,
\end{eqnarray*}
 % -------------- %
the coupling condition yield the following set of equations,
 % -------------- %
\begin{eqnarray*}
  -a\sin{k}-ak\cos{k}+d-ika\cos{k}+ik^2a\sin{k}+ikc = 0\,,\\
  -d+e(\sin{k}-k\cos{k})+ikc+ike(-\cos{k}-k\sin{k}) = 0\,,\\
  -e(\sin{k}-k\cos{k})-ic\sin{k}-id\cos{k}+ike(-\cos{k}-k\sin{k})-ck\cos{k}+dk\sin{k} = 0\,,\\
  a\sin{k}+ak\cos{k}+ic\sin{k}+id\cos{k}-ika\cos{k}+ik^2a\sin{k}-ck\cos{k}+dk\sin{k} = 0\,,
\end{eqnarray*}
 % -------------- %
and the secular equation is
 % -------------- %
$$
  8k[k^4\cos{k}\sin^2{k}+k^2(-\cos{k}+1+2\cos{k}\sin^2{k})+\cos{k}\sin^2{k}] = 0\,,
$$
 % -------------- %
or alternatively
 % -------------- %
$$
  k\sin^2{\frac{k}{2}}\left[2 k^4 \cos^2{\frac{k}{2}}\cos{k}+k^2\left(1+4\cos^2{\frac{k}{2}}\cos{k}\right)+2\cos^2{\frac{k}{2}}\cos{k}\right] = 0\,.
$$
 % -------------- %
The eigenvalues coming from $\sin{\frac{k}{2}}$ are $k = 2n\pi$, $n\in\mathbb{Z}$, the other eigenvalues approach $k = \pi+2n\pi$ and $k = \frac{\pi}{2}+n\pi$, $n\in\mathbb{Z}$, as $k\to \infty$.

We can summarize the above results in the following theorem:

 % -------------- %
\begin{theorem}
The (square roots of) eigenvalues for the octahedron with the preferred-orientation coupling are at $k=2\pi n$ (with multiplicity eight) and $k = \pm \frac{2}{3}\pi+2\pi n$ (with multiplicity two) with $n\in \mathbb{Z}$; the other (square roots of) eigenvalues are situated in the intervals
 % -------------- %
\begin{eqnarray*}
  k \in \left[\pi +2\pi n-\frac{\sqrt{10}}{k}+\mathcal{O}\left(\frac{1}{k^2}\right), \pi +2\pi n+\frac{\sqrt{10}}{k}+\mathcal{O}\left(\frac{1}{k^2}\right)\right]\,,\\
  k \in \left[\frac{\pi}{2} +\pi n-\frac{5}{k^2}+\mathcal{O}\left(\frac{1}{k^4}\right), \frac{\pi}{2} +\pi n+\frac{5}{k^2}+\mathcal{O}\left(\frac{1}{k^4}\right)\right]\,.\\
\end{eqnarray*}
 % -------------- %
\end{theorem}
 % -------------- %
\begin{proof}
Counting the contributions from the component operators one can establish the multiplicities of the eigenvalues with the exactly identified positions. It remains to prove that the (square roots of) the other eigenvalues are situated in the indicated intervals. Let us focus on the eigenvalues with $k \to \pi +2n\pi$, $n\in \mathbb{Z}$, first. For $j = 0$ we have from \eqref{eq:8sten:et:seceq}
 % -------------- %
$$
 \left|\cos{\frac{k}{2}}\right|\leq \sqrt{\frac{3}{2}}\frac{1}{k}\quad \Rightarrow\quad |k-\pi -2n\pi|\leq \frac{\sqrt{6}}{k}+\mathcal{O}\left(\frac{1}{k^2}\right)\,.
$$
 % -------------- %
Similarly, for $j = 1,3$ we get
 % -------------- %
$$
  \left|\cos{k}\cos^2{\frac{k}{2}}\right|\leq \frac{5}{2k^2}+\mathcal{O}\left(\frac{1}{k^4}\right) \quad \Rightarrow\quad |k-\pi -2n\pi|\leq \frac{\sqrt{10}}{k}+\mathcal{O}\left(\frac{1}{k^2}\right)\,.
$$
 % -------------- %
Let us now turn to the eigenvalues with $k\to \frac{\pi}{2}+n\pi$, $n\in \mathbb{Z}$. For $j = 1,3$ we have
 % -------------- %
$$
  \left|\cos{k}\cos^2{\frac{k}{2}}\right|\leq \frac{5}{2k^2}+\mathcal{O}\left(\frac{1}{k^4}\right) \quad \Rightarrow\quad |k-\frac{\pi}{2}-n\pi|\leq \frac{5}{k^2}+\mathcal{O}\left(\frac{1}{k^4}\right)\,,
$$
 % -------------- %
while for $j = 2$ we get
 % -------------- %
$$
  \left|\cos^2{\frac{k}{2}}-\frac{1}{2}\right|\leq \frac{1}{k^2}+\mathcal{O}\left(\frac{1}{k^4}\right) \quad \Rightarrow\quad |k-\frac{\pi}{2}-n\pi|\leq \frac{2}{k^2}+\mathcal{O}\left(\frac{1}{k^4}\right)\,,
$$
 % -------------- %
which concludes the proof.
\end{proof}

We see that the octahedron quantum graph with the preferred-orientation coupling differs from the previous (and the following) cases: its spectrum has at high energies an infinite component different from the one with the Dirichlet Laplacian asymptotics.

%%%%%%%%%%%%%%%
\section{The dodecahedron}
%%%%%%%%%%%%%%%

It remains to inspect the last two Platonic solids. The procedure would be the same as above so we could be briefer in explaining the argument. Let us first consider the dodecahedron. It has a symmetry given by the 5-fold axis going through centers of two opposite faces. The operator of a clockwise rotation by $\frac{2\pi}{5}$ is $T_5$, and we again employ its eigenvalues $\omega_j = \mathrm{e}^{\frac{2\pi i}{5}j}$, $j = 0,1,2,3,4$, together with the corresponding eigensubspaces to decompose the Hamiltonian. The part of the wavefunction components are showed in Figure~\ref{fig7} (the vertex $A$ belongs to the upper face and the vertex $D$ to the lower face).

 % -------------- %
\begin{figure}
\centering
\includegraphics[height=5cm]{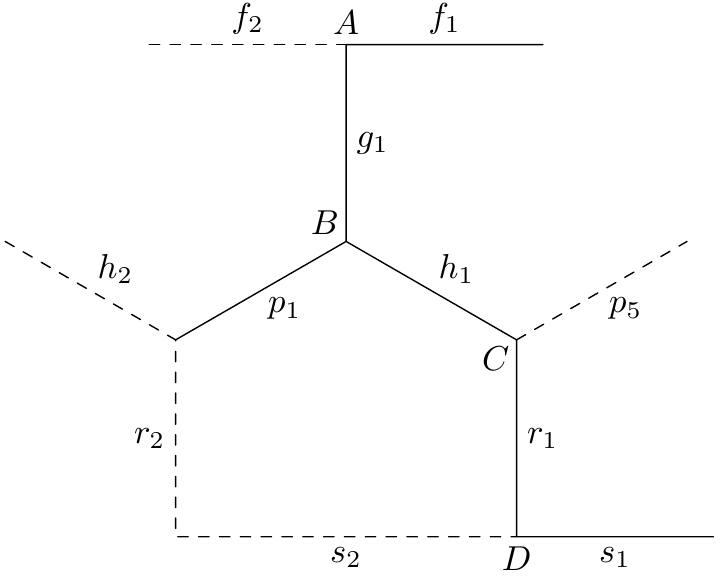}
\caption{A part of an dodecahedron}
\label{fig7}
\end{figure}
 % -------------- %

Similarly to the previous Platonic graphs, we have
 % -------------- %
\begin{eqnarray}
  f_2(x) = \omega_j f_1(x)\,,\quad g_2(x) = \omega_j g_1(x)\,,\quad h_2(x) = \omega_j h_1(x)\,,\quad p_2(x)=\omega_j p_1(x)\,,\nonumber\\
  r_2(x) = \omega_j r_1(x)\,,\quad s_2(x) = \omega_j s_1(x)\,,\quad p_5(x) = \omega_j^4 p_1(x)\,.\label{eq:12sten:01}
\end{eqnarray}
 % -------------- %

%%%%%%%%%%%%%%%
\subsection{$\delta$-condition}
%%%%%%%%%%%%%%%

Following the same pattern, we consider at the vertex $A$ (with $x=0$ for $f_1$ and $g_1$, $x=1$ for $f_2$) the conditions
 % -------------- %
$$
  f_1(0) = f_2(1) = g_1(0)\,,\quad f_1'(0)-f_2'(1)+g_1'(0) = \alpha f_1(0)\,.
$$
 % -------------- %
Using \eqref{eq:12sten:01} we get
 % -------------- %
\begin{equation}
  f_1(0) = \omega_j f_1(1) = g_1(0)\,,\quad f_1'(0)-\omega_j f_1'(1)+g_1'(0) = \alpha f_1(0)\,.\label{eq:12sten:delta:a}
\end{equation}
 % -------------- %
Furthermore, at the vertex $B$ (with $x=0$ for $h_1$ and $p_1$, $x=1$ for $g_1$) the coupling condition is
 % -------------- %
\begin{equation}
  g_1(1) = h_1(0) = p_1(0)\,,\quad -g_1'(1) + h_1'(0)+p_1'(0) = \alpha h_1(0)\,.\label{eq:12sten:delta:b}
\end{equation}
 % -------------- %
At the vertex $C$ (with $x=0$ for $r_1$, $x=1$ for $h_1$ and $p_5$) we obtain the condition
 % -------------- %
$$
  h_1(1) = p_5(1) = r_1(0)\,,\quad -h_1'(1)-p_5'(1)+r_1'(0) = \alpha r_1(0)\,.
$$
 % -------------- %
Using \eqref{eq:12sten:01} we can write
 % -------------- %
\begin{equation}
  h_1(1) = \omega_j^4 p_1(1) = r_1(0)\,,\quad -h_1'(1)-\omega_j^4 p_1'(1)+r_1'(0) = \alpha r_1(0)\,.\label{eq:12sten:delta:c}
\end{equation}
 % -------------- %
Next, at the vertex $D$ (with $x=0$ for $s_1$, $x=1$ for $r_1$ and $s_2$) the coupling is
 % -------------- %
$$
  r_1(1) = s_1(0) = s_2(1)\,,\quad -r_1'(1)+s_1'(0)-s_2'(1) = \alpha s_1(0)\,,
$$
 % -------------- %
and using \eqref{eq:12sten:01} we get
 % -------------- %
\begin{equation}
  r_1(1) = s_1(0) = \omega_j s_1(1)\,,\quad -r_1'(1)+s_1'(0)-\omega_j s_1'(1) = \alpha s_1(0)\,.\label{eq:12sten:delta:d}
\end{equation}
 % -------------- %
Hence we have five component operators supported by the graph shown in Figure~\ref{fig8}. At the vertex $A$ we have eq.~\eqref{eq:12sten:delta:a}, at $B$ eq.~\eqref{eq:12sten:delta:b}, at $C$ eq.~\eqref{eq:12sten:delta:c}, and at $D$ eq.~\eqref{eq:12sten:delta:d}, all with $\omega_j = \mathrm{e}^{\frac{2\pi i}{5}j}$, $j = 0,1,2,3,4$.

 % -------------- %
\begin{figure}
\centering
\includegraphics[width=\textwidth]{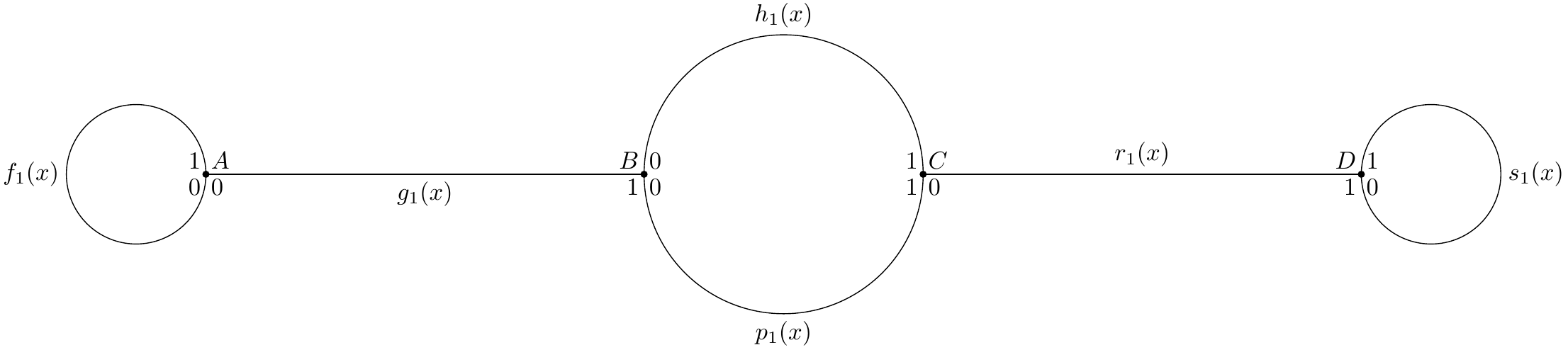}
\caption{The graph supporting component operators for the dodecahedron}
\label{fig8}
\end{figure}
 % -------------- %
The way to obtain the secular equation is the same as for the previous graphs leading to
 % -------------- %
\begin{multline*}
  k^4\sin^2{k}[-9\omega_j^4\sin^2{k}+6\omega_j^4\cos{k}-9\omega_j^3\sin^2{k}+54\omega_j^2\cos{k}\sin^2{k}+81\omega_j\sin^4{k}+6\omega_j^4+6\omega_j^3\cos{k}+9\omega_j^2\sin^2{k}+6\omega_j^3-
\\
-36\omega_j^2\cos{k}-144\omega_j\sin^2{k}+54\cos{k}\sin^2{k}-12\omega_j^2+12\omega_j\cos{k}+9\sin^2{k}+60\omega_j-36\cos{k}-12]
\\
+\alpha k^3\sin^3{k}[6\omega_j^4\cos{k}+2\omega_j^4+6\omega_j^3\cos{k}+54\omega_j^2\sin^2{k}-108\omega_j\cos{k}\sin^2{k}+2\omega_j^3-6\omega_j^2\cos{k}-48\omega_j^2
\\
+96\omega_j\cos{k}+54\sin^2{k}+4\omega_j-6\cos{k}-48]+\alpha^2k^2\sin^4{k}[\omega_j^4+\omega_j^3-18\omega_j^2\cos{k}-54\omega_j\sin^2{k}-\omega_j^2+52\omega_j
\\
-18\cos{k}-1]-2\alpha^3 k\sin^5{k}(\omega_j^2-6\omega_j+1)+\alpha^4\omega_j\sin^6{k} = 0\,.
\end{multline*}
 % -------------- %
Clearly there are eigenvalues with $k=n\pi$, $n\in\mathbb{Z}$, of multiplicity two. The leading term of the secular equation coming from coefficient of $k^4$ can be rewritten using $\sum_{s=0}^4 \omega_j^s = 0$ as
 % -------------- %
$$
  \sin^2{k}(3\cos{k}-1)(3\cos{k}+2)\cos{k}(3\omega_j\cos{k}-2\omega_j^2-\omega_j-2)\,,
$$
 % -------------- %
hence the eigenvalues approach in the $k$-plane the values with $\cos{k} = \frac{1}{3}$ ($k \approx \pm 1.231+2n\pi$, $n\in \mathbb{Z}$), $\cos{k} = -\frac{2}{3}$ ($k \approx \pm 2.301+2n\pi$, $n\in \mathbb{Z}$), $\cos{k} = 0$ ($k = \frac{\pi}{2}+n\pi$, $n\in\mathbb{Z}$), and $\cos{k} = \frac{1}{3}(1+4\mathrm{Re\,}\omega_j)$ (for $j=0$ we do not have a solution; for $j = 1$ and $j = 4$ we have $\cos{k} = \frac{\sqrt{5}}{3}$, i.e. $k \approx \pm 0.730+2n\pi$, $n\in\mathbb{Z}$; for $j = 2$ and $j = 3$ we have $\cos{k} = -\frac{\sqrt{5}}{3}$, i.e. $k \approx \pm 2.412+2n\pi$, $n\in\mathbb{Z}$).

%%%%%%%%%%%%%%%
\subsection{Preferred-orientation coupling condition}
%%%%%%%%%%%%%%%

In that case we have at $A$ the coupling condition
 % -------------- %
$$
  \begin{pmatrix}-1 & 1 & 0\\ 0 & -1& 1\\ 1 & 0 & -1 \end{pmatrix}\begin{pmatrix}f_1(0)\\g_1(0) \\ f_2(1)\end{pmatrix} + i \begin{pmatrix}1 & 1 & 0\\ 0 & 1& 1\\ 1 & 0 & 1 \end{pmatrix}\begin{pmatrix}f_1'(0)\\g_1'(0) \\ -f_2'(1)\end{pmatrix} =0\,.
$$
 % -------------- %
Using \eqref{eq:12sten:01} we get
 % -------------- %
\begin{equation}
  \begin{pmatrix}-1 & 1 & 0\\ 0 & -1& 1\\ 1 & 0 & -1 \end{pmatrix}\begin{pmatrix}f_1(0)\\g_1(0) \\ \omega_j f_1(1)\end{pmatrix} + i \begin{pmatrix}1 & 1 & 0\\ 0 & 1& 1\\ 1 & 0 & 1 \end{pmatrix}\begin{pmatrix}f_1'(0)\\g_1'(0) \\ -\omega_j f_1'(1)\end{pmatrix} =0\,.\label{eq:12sten:et:a}
\end{equation}
 % -------------- %
At the vertex $B$ the coupling condition reads
 % -------------- %
\begin{equation}
  \begin{pmatrix}-1 & 1 & 0\\ 0 & -1& 1\\ 1 & 0 & -1 \end{pmatrix}\begin{pmatrix}g_1(1)\\h_1(0) \\ p_1(0)\end{pmatrix} + i \begin{pmatrix}1 & 1 & 0\\ 0 & 1& 1\\ 1 & 0 & 1 \end{pmatrix}\begin{pmatrix}-g_1'(1)\\h_1'(0) \\ p_1'(0)\end{pmatrix} =0\,.\label{eq:12sten:et:b}
\end{equation}
 % -------------- %
Furthermore, at $C$ we have
 % -------------- %
$$
  \begin{pmatrix}-1 & 1 & 0\\ 0 & -1& 1\\ 1 & 0 & -1 \end{pmatrix}\begin{pmatrix}h_1(1)\\p_5(1) \\ r_1(0)\end{pmatrix} + i \begin{pmatrix}1 & 1 & 0\\ 0 & 1& 1\\ 1 & 0 & 1 \end{pmatrix}\begin{pmatrix}-h_1'(1)\\-p_5'(1) \\ r_1'(0)\end{pmatrix} =0\,;
$$
 % -------------- %
and using \eqref{eq:12sten:01} we find
 % -------------- %
\begin{equation}
  \begin{pmatrix}-1 & 1 & 0\\ 0 & -1& 1\\ 1 & 0 & -1 \end{pmatrix}\begin{pmatrix}h_1(1)\\\omega_j^4 p_1(1) \\ r_1(0)\end{pmatrix} + i \begin{pmatrix}1 & 1 & 0\\ 0 & 1& 1\\ 1 & 0 & 1 \end{pmatrix}\begin{pmatrix}-h_1'(1)\\-\omega_j^4 p_1'(1) \\ r_1'(0)\end{pmatrix} =0\,.\label{eq:12sten:et:c}
\end{equation}
 % -------------- %
Finally, at $D$ we have
 % -------------- %
$$
  \begin{pmatrix}-1 & 1 & 0\\ 0 & -1& 1\\ 1 & 0 & -1 \end{pmatrix}\begin{pmatrix}r_1(1)\\s_1(0) \\ s_2(1)\end{pmatrix} + i \begin{pmatrix}1 & 1 & 0\\ 0 & 1& 1\\ 1 & 0 & 1 \end{pmatrix}\begin{pmatrix}-r_1'(1)\\s_1'(0) \\ -s_2'(1)\end{pmatrix} =0\,,
$$
which in view of \eqref{eq:12sten:01} yields
 % -------------- %
\begin{equation}
  \begin{pmatrix}-1 & 1 & 0\\ 0 & -1& 1\\ 1 & 0 & -1 \end{pmatrix}\begin{pmatrix}r_1(1)\\s_1(0) \\ \omega_j s_1(1)\end{pmatrix} + i \begin{pmatrix}1 & 1 & 0\\ 0 & 1& 1\\ 1 & 0 & 1 \end{pmatrix}\begin{pmatrix}-r_1'(1)\\s_1'(0) \\ -\omega_j s_1'(1)\end{pmatrix} =0\,.\label{eq:12sten:et:d}
\end{equation}
 % -------------- %
The analysis then reduces to the investigation of five component operators supported by the graph in Figure~\ref{fig8} with the coupling conditions \eqref{eq:12sten:et:a}, \eqref{eq:12sten:et:b}, \eqref{eq:12sten:et:c} and \eqref{eq:12sten:et:d} and with $\omega_j = \mathrm{e}^{\frac{2\pi i}{5}j}$. The secular equation then takes the following form,
 % -------------- %
\begin{multline*}
  \omega_j k^6\sin^6{k}+k^4\sin^4{k}[\omega_j^4+\omega_j^3+2\omega_j^2\cos{k}-\omega_j^2-12\omega_j\cos^2{k}+2\cos{k}-1]
\\
+k^2\sin^2{k}[\omega_j^4(5\sin^2{k}-2\cos{k}-6)+\omega_j^3(5\sin^2{k}-2\cos{k}-6)+\omega_j^2(18\cos{k}\sin^2{k}-5\sin^2{k}-12\cos{k}+4)
\\
+\omega_j(54\sin^4{k}-88\sin^2{k}-4\cos{k}+36)+18\cos{k}\sin^2{k}-5\sin^2{k}-12\cos{k}+4]
\\+  \omega_j^4(3\sin^4{k}-4\cos{k}\sin^2{k})+\omega_j^3(3\sin^4{k}-4\cos{k}\sin^2{k})+\omega_j^2(54\cos{k}\sin^4{k}-3\sin^4{k}
\\
-48\cos{k}\sin^2{k}+8\sin^2{k}+8\cos{k}-8)+4\omega_j(27\sin^6{k}-51\sin^4{k}-2\cos{k}\sin^2{k}+28\sin^2{k}+4\cos{k}-4)
\\
+54\cos{k}\sin^4{k}-3\sin^4{k}-48\cos{k}\sin^2{k}+8\sin^2{k}+8\cos{k}-8 = \mathcal{O}(k^{-2})\,.
\end{multline*}
 % -------------- %

 % -------------- %
\begin{theorem}
For the dodecahedron with the preferred-orientation coupling the (square roots of) eigenvalues $k$ are for large $k$ situated in the intervals
 % -------------- %
$$
  k \in \left(n\pi -\frac{5.51}{k}+\mathcal{O}\left(\frac{1}{k^2}\right),n\pi+\frac{5.51}{k}+\mathcal{O}\left(\frac{1}{k^2}\right)\right)\,,\quad n\in\mathbb{Z}\,.
$$
 % -------------- %
\end{theorem}
 % -------------- %
\begin{proof}
The problem reduces to the cubic equation $y^3+a_2 y^2+a_1 y+ a_0 = 0$ in the variable $y = k^2\sin^2{k}$ with the bounds on its coefficients $|a_2|\leq 20$, $|a_1|\leq 286$, and $|a_0|\leq 736 + \mathcal{O}(k^{-2})$. Using Fujiwaras's bound \cite{Mar} on the positions of the roots of such an equation we get
 % -------------- %
$$
  |y|\leq 2 \mathrm{max\,}\left(|a_2|,|a_1|^{1/2},\left|\frac{a_0}{2}\right|^{1/3}\right) = 40+\mathcal{O}(k^{-2})\,.
$$
 % -------------- %
Furthermore, we use repeatedly the following bound on the positions of the roots,
 % -------------- %
$$
  |y|\leq \sqrt[3]{|a_2| a^2 + |a_1| a + |a_0|}\,,
$$
 % -------------- %
where $a$ is the previously known bound, and in this way we obtain $|y|\leq 30.3+\mathcal{O}(k^{-2})$. Hence we get $|k\sin{k}|\leq 5.51 + \mathcal{O}(k^{-2})$, from which the claim of the theorem follows.
\end{proof}

%%%%%%%%%%%%%%%
\section{Icosahedron}
%%%%%%%%%%%%%%%

Finally, for an icosahedron the symmetry refers to the 5-fold axis going through two opposite vertices, $A$ and $D$; a part of the graph is shown in Figure~\ref{fig9}. The operator of a clockwise rotation by $\frac{2\pi}{5}$, as seen from the exterior of the vertex $A$, is denoted $T_5$ and its eigenvalues are $\omega_j = \mathrm{e}^{\frac{2\pi i}{5}j}$, $j=0,1,2,3,4$. The corresponding eigensubspaces give rise to component operators supported on the sketched graph.
 % -------------- %
\begin{figure}
\centering
\includegraphics[height=8cm]{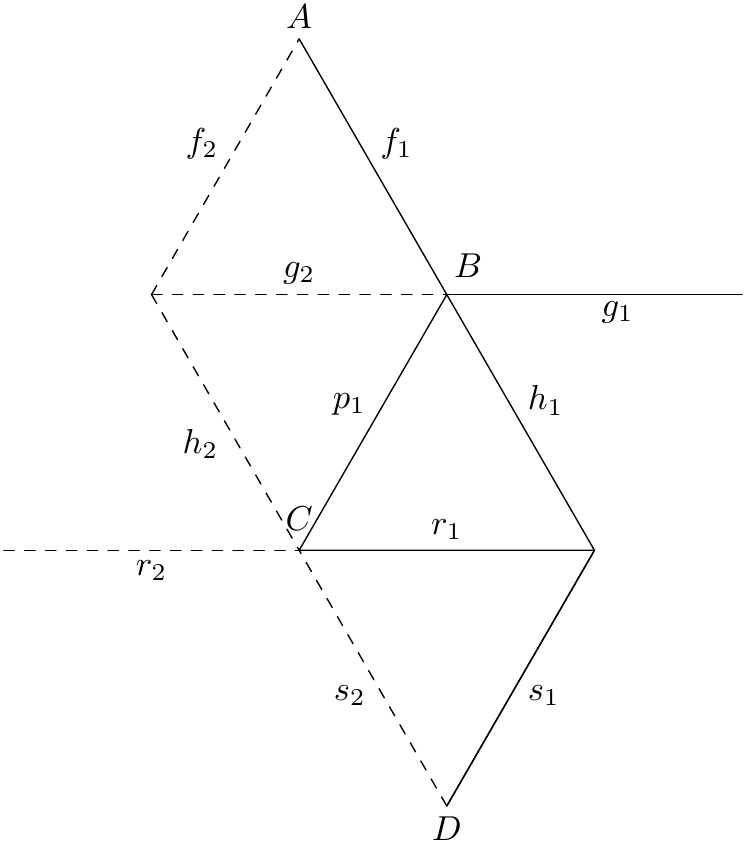}
\caption{A part of an icosahedron}
\label{fig9}
\end{figure}
 % -------------- %
We have
 % -------------- %
\begin{eqnarray}
  f_2(x) = \omega_j f_1(x)\,,\quad g_2(x) = \omega_j g_1(x)\,,\quad h_2(x) = \omega_j h_1(x)\,,\nonumber\\
  p_2(x) = \omega_j p_1(x)\,,\quad r_2(x) = \omega_j r_1(x)\,,\quad s_2(x) = \omega_j s_1(x)\,.\label{eq:20sten:01}
\end{eqnarray}
 % -------------- %

%%%%%%%%%%%%%%%
\subsection{$\delta$-condition}
%%%%%%%%%%%%%%%

At the vertex $A$ (with $x=0$ for $f_1$, $f_2$, $f_3$, $f_4$ and $f_5$) we have
 % -------------- %
$$
  f_1(0) = f_2(0) = f_3(0) = f_4(0) = f_5(0)\,,\quad f_1'(0)+f_2'(0)+f_3'(0)+f_4'(0)+f_5'(0) = \alpha f_1(0)\,.
$$
 % -------------- %
Using \eqref{eq:20sten:01} we find
 % -------------- %
\begin{equation}
  f_1'(0) = \frac{\alpha}{5}f_1(0)\,,\quad j=0\,,\label{eq:20sten:delta:a0}
\end{equation}
 % -------------- %
and
 % -------------- %
\begin{equation}
  f_1(0)=0\,,\quad j=1,2,3,4\,.\label{eq:20sten:delta:a1}
\end{equation}
 % -------------- %
At the vertex $B$ (corresponding to $x=0$ for $g_1$, $h_1$ and $p_1$, $x=1$ for $f_1$ and $g_2$) we have
 % -------------- %
$$
  f_1(1)=g_1(0) = g_2(1)=h_1(0) = p_1(0)\,,\quad -f_1'(1)+g_1'(0)-g_2'(1)+h_1'(0)+p_1'(0) = \alpha g_1(0)\,,
$$
 % -------------- %
which with the help of \eqref{eq:20sten:01} yields
 % -------------- %
\begin{equation}
  f_1(1)=g_1(0) = \omega_j g_1(1)=h_1(0) = p_1(0)\,,\quad -f_1'(1)+g_1'(0)-\omega_j g_1'(1)+h_1'(0)+p_1'(0) = \alpha g_1(0)\,.\label{eq:20sten:delta:b}
\end{equation}
 % -------------- %
Furthermore, at the vertex $C$ (with $x=0$ for $r_1$ and $s_2$, $x=1$ for $p_1$, $r_2$, and $h_2$) we get
 % -------------- %
$$
  p_1(1) = r_1(0) = s_2(0) = r_2(1) = h_2(1)\,,\quad -p_1'(1)+r_1'(0)+s_2'(0)-r_2'(1)-h_2'(1) = \alpha r_1(0)\,,
$$
 % -------------- %
so that \eqref{eq:20sten:01} leads to
 % -------------- %
\begin{equation}
  p_1(1) = r_1(0) = \omega_j s_1(0) =  \omega_j r_1(1) =  \omega_j h_1(1)\,,\quad -p_1'(1)+r_1'(0)+ \omega_j s_1'(0)- \omega_j r_1'(1)- \omega_j h_1'(1) = \alpha r_1(0)\,.\label{eq:20sten:delta:c}
\end{equation}
 % -------------- %
Finally, at $D$ (referring to $x = 1$ for $s_1$, $s_2$, $s_3$, $s_4$, and $s_5$) we can write
 % -------------- %
$$
  s_1(1) = s_2(1) = s_3(1) = s_4(1) = s_5(1)\,,\quad -s_1'(1)-s_2'(1)-s_3'(1)-s_4'(1)-s_5'(1)=\alpha s_1(1)\,,
$$
 % -------------- %
and using \eqref{eq:20sten:01} we arrive at
 % -------------- %
\begin{eqnarray}
  -s_1'(1) = \frac{\alpha}{5}s_1(1)\,,\quad j=0\,,\label{eq:20sten:delta:d0}\\
  s_1(1) = 0\,,\quad j=1,2,3,4\,.\label{eq:20sten:delta:d1}
\end{eqnarray}
 % -------------- %
We have five component operatots on the graph from Figure~\ref{fig10}. The first one has Robin conditions~\eqref{eq:20sten:delta:a0} and \eqref{eq:20sten:delta:d0} at the vertices $A$ and $D$, and the $\delta$-conditions \eqref{eq:20sten:delta:b} and \eqref{eq:20sten:delta:c} at $B$ and $C$, respectively, both with $\omega_j = 1$. The other component operators have Dirichlet conditions at $A$ and $D$~\eqref{eq:20sten:delta:a1}, and \eqref{eq:20sten:delta:d1} and the conditions \eqref{eq:20sten:delta:b} and \eqref{eq:20sten:delta:c} at $B$ and $C$ with $\omega_j = \mathrm{e}^{\frac{2\pi i}{5}j}$, $j = 1, 2, 3 , 4$.

 % -------------- %
\begin{figure}
\centering
\includegraphics[width=\textwidth]{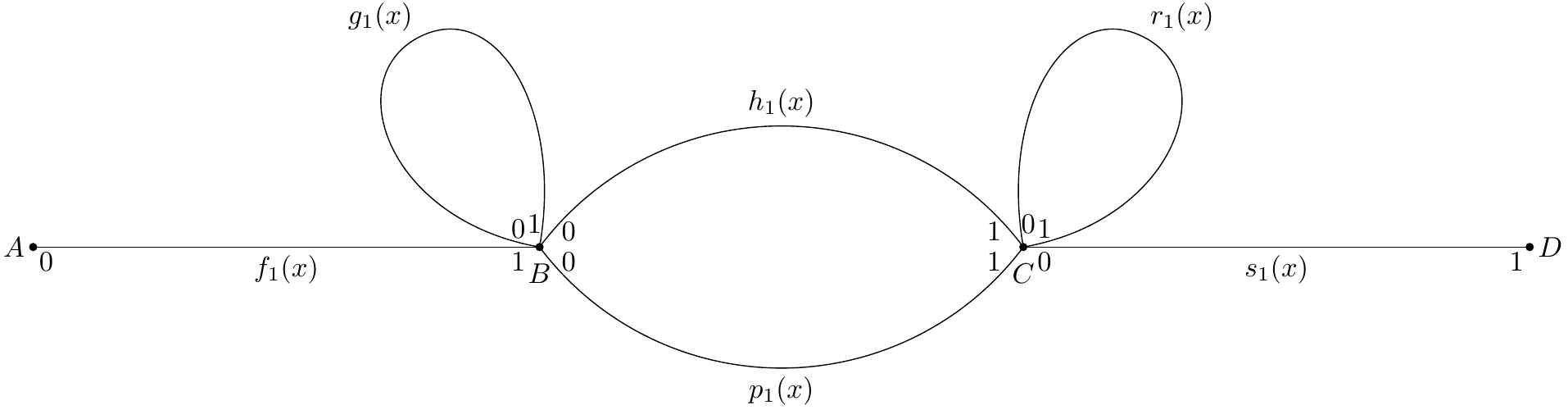}
\caption{The graph supporting component operators for the icosahedron}
\label{fig10}
\end{figure}
 % -------------- %
The secular equation for $j = 0$ is
 % -------------- %
\begin{multline*}
  k^4\sin^2{k} (-625 \sin^4{k}-500\cos{k}\sin^2{k}+1000\sin^2{k}+400\cos{k}-400)+
\\
+\alpha k^3\sin^3{k}(500\cos{k}\sin^2{k}-300\sin^2{k}-400\cos{k}+280)+
\\
+\alpha^2 k^2\sin^4{k}(150\sin^2{k}+60\cos{k}-140)+\alpha^3 k \sin^5{k}(-20\cos{k}+4)-\alpha^4\sin^6{k} = 0
\end{multline*}
 % -------------- %
giving obviously rise to the eigenvalues with $k = \pi n$, $n\in\mathbb{Z}$ of multiplicity two. The remaining part of the leading term can be written as an equation of order four in $\cos{k}$,
 % -------------- %
$$
  -25\cos^4{k}+20\cos^3{k}+10\cos^2{k}-4\cos{k}-1\,,
$$
 % -------------- %
which has four real roots $\cos{k} = 1$, which corresponds to $k = 2n\pi$, $n\in\mathbb{Z}$, $\cos{k} = -\frac{1}{5}$ correspondong to $k\approx \pm 1.772 +2n\pi$, $n\in\mathbb{Z}$,  $\cos{k} = \frac{\sqrt{5}}{5}$ referring to $k\approx \pm 1.107 +2n\pi$, $n\in\mathbb{Z}$, and finally, $\cos{k} = -\frac{\sqrt{5}}{5}$ which corresponds to $k\approx \pm 2.034 +2n\pi$, $n\in\mathbb{Z}$. The square roots of the eigenvalues of the system approach these values in the high-energy region.

For $j = 1, 2, 3, 4$ we obtain the secular equation of the form
 % -------------- %
\begin{multline*}
 k^2 \sin^4{k}(25\omega_j^4\cos^2{k}-10\omega_j^3\cos{k}-\omega_j^3+\omega_j^2+\omega_j-10\cos{k}-1) +
\\
+k\alpha \sin^5{k}(10\omega_j^4\cos{k}-2\omega_j^3-2)+\alpha^2\omega_j^4\sin^6{k}\,.
\end{multline*}
 % -------------- %
solved by the eigenvalues with $k = n\pi$, $n\in\mathbb{Z}$, of multiplicity four, and the other ones for which the leading term can be written as
 % -------------- %
$$
  (5\cos{k}+1)(5\omega_j^4\cos{k}-\omega_j^4-2\omega_j^3-2)\,,
$$
 % -------------- %
giving rise to the roots $\cos{k} = -\frac{1}{5}$ corresponding to $k\approx \pm 1.772 +2n\pi$, $n\in\mathbb{Z}$, and $\cos{k}= \frac{1}{5}(1+4\mathrm{Re\,}\omega_j)$, in particular for $j = 1,4$ the latter term is equal to $\cos{k} = \frac{\sqrt{5}}{5}$ referring to $k\approx \pm 1.107 +2n\pi$, $n\in\mathbb{Z}$, and for $j = 2,3$ we have $\cos{k} = -\frac{\sqrt{5}}{5}$, i.e. $k\approx \pm 2.034 +2n\pi$, $n\in\mathbb{Z}$. The (square roots of) the eigenvalues approach these values at high energies.

%%%%%%%%%%%%%%%
\subsection{Preferred-orientation coupling}
%%%%%%%%%%%%%%%

Finally, in this case we have at $A$
 % -------------- %
$$
  \begin{pmatrix}-1& 1 & 0 & 0 &0\\ 0 & -1 & 1 & 0 &0\\ 0 & 0 & -1 & 1 & 0\\ 0 & 0 & 0 & -1 & 1\\ 1 & 0 & 0 & 0 & -1\end{pmatrix} \begin{pmatrix}f_1(0)\\f_2(0)\\f_3(0)\\f_4(0)\\f_5(0)\end{pmatrix} + i  \begin{pmatrix}1& 1 & 0 & 0 &0\\ 0 & 1 & 1 & 0 &0\\ 0 & 0 & 1 & 1 & 0\\ 0 & 0 & 0 & 1 & 1\\ 1 & 0 & 0 & 0 & 1\end{pmatrix} \begin{pmatrix}f_1'(0)\\f_2'(0)\\f_3'(0)\\f_4'(0)\\f_5'(0)\end{pmatrix} = 0\,,
$$
 % -------------- %
so using \eqref{eq:20sten:01} we obtain
 % -------------- %
\begin{eqnarray}
  f_1'(0)= 0\,,\quad j=0\,,\label{eq:20sten:et:a0}\\
  f_1'(0) = -\sqrt{5-2\sqrt{5}}f_1(0)\,,\quad j=1\,,\label{eq:20sten:et:a1}\\
  f_1'(0) = -\sqrt{5+2\sqrt{5}}f_1(0)\,,\quad j=2\,,\label{eq:20sten:et:a2}\\
  f_1'(0) = \sqrt{5+2\sqrt{5}}f_1(0)\,,\quad j=3\,,\label{eq:20sten:et:a3}\\
  f_1'(0) = \sqrt{5-2\sqrt{5}}f_1(0)\,,\quad j=4\,.\label{eq:20sten:et:a4}
\end{eqnarray}
 % -------------- %
At the vertex $B$ the condition reads
 % -------------- %
$$
  \begin{pmatrix}-1& 1 & 0 & 0 &0\\ 0 & -1 & 1 & 0 &0\\ 0 & 0 & -1 & 1 & 0\\ 0 & 0 & 0 & -1 & 1\\ 1 & 0 & 0 & 0 & -1\end{pmatrix} \begin{pmatrix}f_1(1)\\g_1(0)\\h_1(0)\\p_1(0)\\g_2(1)\end{pmatrix} + i  \begin{pmatrix}1& 1 & 0 & 0 &0\\ 0 & 1 & 1 & 0 &0\\ 0 & 0 & 1 & 1 & 0\\ 0 & 0 & 0 & 1 & 1\\ 1 & 0 & 0 & 0 & 1\end{pmatrix} \begin{pmatrix}-f_1'(1)\\g_1'(0)\\h_1'(0)\\p_1'(0)\\-g_2'(1)\end{pmatrix} = 0\,,
$$
 % -------------- %
which by means of \eqref{eq:20sten:01} yields
 % -------------- %
\begin{equation}
  \begin{pmatrix}-1& 1 & 0 & 0 &0\\ 0 & -1 & 1 & 0 &0\\ 0 & 0 & -1 & 1 & 0\\ 0 & 0 & 0 & -1 & 1\\ 1 & 0 & 0 & 0 & -1\end{pmatrix} \begin{pmatrix}f_1(1)\\g_1(0)\\h_1(0)\\p_1(0)\\\omega_j g_1(1)\end{pmatrix} + i  \begin{pmatrix}1& 1 & 0 & 0 &0\\ 0 & 1 & 1 & 0 &0\\ 0 & 0 & 1 & 1 & 0\\ 0 & 0 & 0 & 1 & 1\\ 1 & 0 & 0 & 0 & 1\end{pmatrix} \begin{pmatrix}-f_1'(1)\\g_1'(0)\\h_1'(0)\\p_1'(0)\\-\omega_j g_1'(1)\end{pmatrix} = 0\,.\label{eq:20sten:et:b}
\end{equation}
 % -------------- %
Furthermore, the coupling condition at $C$ is
 % -------------- %
$$
  \begin{pmatrix}-1& 1 & 0 & 0 &0\\ 0 & -1 & 1 & 0 &0\\ 0 & 0 & -1 & 1 & 0\\ 0 & 0 & 0 & -1 & 1\\ 1 & 0 & 0 & 0 & -1\end{pmatrix} \begin{pmatrix}p_1(1)\\r_1(0)\\s_2(0)\\r_2(1)\\h_2(1)\end{pmatrix} + i  \begin{pmatrix}1& 1 & 0 & 0 &0\\ 0 & 1 & 1 & 0 &0\\ 0 & 0 & 1 & 1 & 0\\ 0 & 0 & 0 & 1 & 1\\ 1 & 0 & 0 & 0 & 1\end{pmatrix} \begin{pmatrix}-p_1'(1)\\r_1'(0)\\s_2'(0)\\-r_2'(1)\\-h_2'(1)\end{pmatrix} = 0\,,
$$
 % -------------- %
rewritten with the help of \eqref{eq:20sten:01} as
 % -------------- %
\begin{equation}
  \begin{pmatrix}-1& 1 & 0 & 0 &0\\ 0 & -1 & 1 & 0 &0\\ 0 & 0 & -1 & 1 & 0\\ 0 & 0 & 0 & -1 & 1\\ 1 & 0 & 0 & 0 & -1\end{pmatrix} \begin{pmatrix}p_1(1)\\r_1(0)\\\omega_j s_1(0)\\ \omega_j r_1(1)\\ \omega_j h_1(1)\end{pmatrix} + i  \begin{pmatrix}1& 1 & 0 & 0 &0\\ 0 & 1 & 1 & 0 &0\\ 0 & 0 & 1 & 1 & 0\\ 0 & 0 & 0 & 1 & 1\\ 1 & 0 & 0 & 0 & 1\end{pmatrix} \begin{pmatrix}-p_1'(1)\\r_1'(0)\\ \omega_j s_1'(0)\\-\omega_j r_1'(1)\\- \omega_j h_1'(1)\end{pmatrix} = 0\,.\label{eq:20sten:et:c}
\end{equation}
 % -------------- %
Finally, the coupling condition at $D$ is
 % -------------- %
$$
  \begin{pmatrix}-1& 1 & 0 & 0 &0\\ 0 & -1 & 1 & 0 &0\\ 0 & 0 & -1 & 1 & 0\\ 0 & 0 & 0 & -1 & 1\\ 1 & 0 & 0 & 0 & -1\end{pmatrix} \begin{pmatrix}s_5(1)\\s_4(1)\\s_3(1)\\s_2(1)\\s_1(1)\end{pmatrix} + i  \begin{pmatrix}1& 1 & 0 & 0 &0\\ 0 & 1 & 1 & 0 &0\\ 0 & 0 & 1 & 1 & 0\\ 0 & 0 & 0 & 1 & 1\\ 1 & 0 & 0 & 0 & 1\end{pmatrix} \begin{pmatrix}-s_5'(1)\\-s_4'(1)\\-s_3'(1)\\-s_2'(1)\\-s_1'(1)\end{pmatrix} = 0\,.
$$
 % -------------- %
and can be written using~\eqref{eq:20sten:01} as
 % -------------- %
\begin{eqnarray}
  s_1'(1)= 0\,,\quad j=0\,,\label{eq:20sten:et:d0}\\
  s_1'(1) = -\sqrt{5-2\sqrt{5}}\,s_1(1)\,,\quad j=1\,,\label{eq:20sten:et:d1}\\
  s_1'(1) = -\sqrt{5+2\sqrt{5}}\,s_1(1)\,,\quad j=2\,,\label{eq:20sten:et:d2}\\
  s_1'(1) = \sqrt{5+2\sqrt{5}}\,s_1(1)\,,\quad j=3\,,\label{eq:20sten:et:d3}\\
  s_1'(1) = \sqrt{5-2\sqrt{5}}\,s_1(1)\,,\quad j=4\,.\label{eq:20sten:et:d4}
\end{eqnarray}
 % -------------- %
We have five component operators on the graph of Figure~\ref{fig10}. The first one has Neumann coupling~\eqref{eq:20sten:et:a0} and~\eqref{eq:20sten:et:d0} at $A$ and $D$, and the coupling~\eqref{eq:20sten:et:b} and \eqref{eq:20sten:et:c} with $\omega_j = 1$ at $B$ and $C$, respectively. The other graphs have the Robin coupling at $A$ and $D$ (by eq.~\eqref{eq:20sten:et:a1}--\eqref{eq:20sten:et:a4} and eq.~\eqref{eq:20sten:et:d1}--\eqref{eq:20sten:et:d4}) and the coupling~\eqref{eq:20sten:et:b} and \eqref{eq:20sten:et:c} with $\omega_j = \mathrm{e}^{\frac{2\pi i}{5}j}$, $j=1,2,3,4$ at $B$ and $C$, respectively.

The secular equation is then of the form
 % -------------- %
\begin{multline*}
  -\omega_j^2 k^6\sin^6{k}+k^4\sin^4{k}[\omega_j^4(19\cos^2{k}+4\cos{k}+1)+2\omega_j^3(21\cos^2{k}+2\cos{k}+1)
\\
+\omega_j^2(19\cos^2{k}+4\cos{k}+1)+2\omega_j(\cos{k}-1)+2\cos{k}-2]+2k^2\sin^2{k}[-\omega_j^4(45\sin^4{k}-20\cos{k}\sin^2{k}
\\
-64\sin^2{k}-12\cos{k}+12)-2\omega_j^3(65\sin^4{k}-8\cos{k}\sin^2{k}-94\sin^2{k}-24\cos{k}+24)-\omega_j^2(45\sin^4{k}
\\
-20\cos{k}\sin^2{k}-64\sin^2{k}-12\cos{k}+12)+2\omega_j(4\cos{k}\sin^2{k}+\sin^2{k}-2\cos{k}+2)+2(4\cos{k}\sin^2{k}
\\
+\sin^2{k}-2\cos{k}+2)]+2[\omega_j^4(5\sin^6{k}+8\cos{k}\sin^4{k}-12\sin^4{k}-8\cos{k}\sin^2{k}+8\sin^2{k})
\\
-2\omega_j^3(105\sin^6{k}+14\cos{k}\sin^4{k}-176\sin^4{k}-80\cos{k}\sin^2{k}+96\sin^2{k}+32\cos{k}-32)
\\
+\omega_j^2(5\sin^6{k}+8\cos{k}\sin^4{k}-12\sin^4{k}-8\cos{k}\sin^2{k}+8\sin^2{k})
\\
-2\omega_j(7\cos{k}\sin^4{k}-4\cos{k}\sin^2{k}+4\sin^2{k})-14\cos{k}\sin^4{k}+8\cos{k}\sin^2{k}-8\sin^2{k}]=\mathcal{O}(k^{-2})\,.
\end{multline*}
 % -------------- %

 % -------------- %
\begin{theorem}
For the icosahedron with the preferred-orientation coupling the (square roots of) eigenvalues $k$ are for large $k$ placed in the intervals
 % -------------- %
$$
  k \in \left(n\pi -\frac{10.84}{k}+\mathcal{O}\left(\frac{1}{k^2}\right),n\pi+\frac{10.84}{k}+\mathcal{O}\left(\frac{1}{k^2}\right)\right)\,,\quad n\in\mathbb{Z}\,.
$$
 % -------------- %
\end{theorem}
 % -------------- %
\begin{proof}
The proof is similar to that for the dodecahedron. We also have the cubic equation $y^3+a_2 y^2+a_1 y+ a_0 = 0$ in the variable $y = k^2\sin^2{k}$ with the bounds on its coefficients $|a_2|\leq 104$, $|a_1|\leq 1544$, $|a_0|\leq 2424 + \mathcal{O}(1/k^2)$. Fujiwara's bound \cite{Mar} on the roots of this equation leads to $|y|\leq 208$. Applying repeatedly the bound
 % -------------- %
$$
  |y|\leq \sqrt[3]{|a_2| a^2 + |a_1| a + |a_0|}\,,
$$
 % -------------- %
where $a$ is the previously known bound, we obtain $|y|\leq 117.4+\mathcal{O}(1/k^2)$, and hence $|k\sin{k}|\leq 10.84+\mathcal{O}(1/k^2)$. From this the claim of the theorem follows.
\end{proof}

We see that for the dodecahedron and icosahedron with preferred-orientation coupling the high eigenvalues cluster again around those of the Dirichlet Laplacian on a unit length interval.

%%%%%%%%%%%%%%%
\section*{Acknowledgements}
%%%%%%%%%%%%%%%
The authors are obliged to Ond\v{r}ej Turek for a useful discussion. The research was supported by the Czech Science Foundation (GA\v{C}R) within the project No. 17-01706S. P.E. also acknowledges support of the EU project CZ.02.1.01/0.0/0.0/16\textunderscore 019/0000778, and J.L. of the research programme ``Mathematical Physics and Differential Geometry'' of the Faculty of Science of the University of Hradec Kr\'alov\'e.

\end{document}